\documentclass[a4paper,12pt,intlimits,oneside]{amsart}
\usepackage{amsmath}
\usepackage{amsthm}
\usepackage{latexsym}
\usepackage{amssymb}
\usepackage{xcolor}
\numberwithin{figure}{section}
\def\R{{\mathbb R}}
\def\C{{\mathbb C}}
\def\T{{\mathbb T}}
\def\Z{{\mathbb Z}}
\def\S{{\mathbb S}}

\def\N{{\mathbb N}}
\def\e{\varepsilon}
\def\s{\vskip 0.25cm\noindent}
\def\build#1_#2^#3{\mathrel{
\mathop{\kern 0pt#1}\limits_{#2}^{#3}}}
\def\td_#1,#2{\mathrel{\mathop{\build\longrightarrow_{#1\rightarrow #2}^{}}}}
\newtheorem{theorem}{Theorem}[section]
\newtheorem{corollary}{Corollary}
\newtheorem{proposition}{Proposition}
\newtheorem{Lemma}{Lemma}
\newtheorem{remark}{Remark}

\begin{document}
\title{Invariant tori for the cubic Szeg\"o equation}
\author{Patrick G\'erard}
\address{Universit\'e Paris-Sud XI, Laboratoire de Math\'ematiques
d'Orsay, CNRS, UMR 8628, et Institut Universitaire de France} \email{{\tt Patrick.Gerard@math.u-psud.fr}}

\author[S. Grellier]{Sandrine Grellier}
\address{F\'ed\'eration Denis Poisson, MAPMO-UMR 6628,
D\'epartement de Math\'ematiques, Universit\'e d'Orleans, 45067
Orl\'eans Cedex 2, France} \email{{\tt
Sandrine.Grellier@univ-orleans.fr}}

\subjclass[2010]{35B15, 37K15, 47B35}

\keywords{} \maketitle
\renewcommand{\abstractname}{R\'esum\'e}
\begin{abstract}
Nous poursuivons l'\'etude de l'\'equation hamiltonienne suivante sur l'espace de
Hardy du cercle
$$i\partial _tu=\Pi(|u|^2u)\ ,$$
o\`u $\Pi$ d\'esigne le projecteur de Szeg\"o. Cette \'equation est
un cas mod\`ele d'\'equation sans aucune propri\' et\' e dispersive. 
Dans un travail pr\'ec\'edent, nous avons montr\'e qu'elle admettait une paire de Lax et 
qu'elle \'etait compl\`etement int\'egrable. Dans cet article, nous construisons  les variables action-angle,
ce qui nous permet de ramener la r\'esolution explicite de l'\'equation \`a un probl\`eme 
de diagonalisation. Une cons\'equence de cette construction est la solution d'un probl\`eme spectral
inverse pour les op\'erateurs de Hankel. Nous \'etablissons \'egalement la stabilit\'e des tores invariants 
correspondants. En outre,  des formules explicites de r\'esolution ainsi obtenues, nous d\'eduisons  la classification 
des ondes progressives orbitalement stables et instables. 
\end{abstract}
\renewcommand{\abstractname}{Abstract}
\begin{abstract}
We continue the study of   the following Hamiltonian equation on the Hardy
space of the circle,
$$i\partial _tu=\Pi(|u|^2u)\ ,$$
where $\Pi $ denotes  the Szeg\"o projector. This equation can be seen as
a toy model for totally non dispersive evolution equations. In a previous work, we
proved that this equation admits a Lax pair,  and that it is completely integrable.
In this paper, we construct the action-angle variables, which reduces the explicit resolution of 
the equation to a diagonalisation problem. As a consequence, we solve an inverse spectral 
problem for Hankel operators. Moreover, we establish the stability 
of the corresponding invariant tori. Furthermore, from the explicit formulae, we deduce
the classification of  orbitally stable and unstable traveling waves.
\end{abstract}
\thanks {The authors would like to thank L.~Baratchart,
T.~Kappeler, S.~Kuksin for valuable discussions. 
They also acknowledge the supports
of the following ANR projects :  EDP dispersives
(ANR-07-BLAN-0250-01) for the first author, and AHPI
(ANR-07-BLAN-0247-01) for the second author.}

\begin{section}{Introduction}
\subsection{The cubic Szeg\"o equation}
In the paper \cite{GG}, we introduced the evolution equation
\begin{equation}\label{szego}
i\partial _tu=\Pi(|u|^2u)\ ,
\end{equation}
posed on the Hardy space of the circle 
$$L^2_+=\{ u\; :\;  u=\sum _{k=0}^\infty \hat u(k)\, {\rm e}^{ik\theta }\ ,\ \sum _{k=0}^\infty |\hat u(k)|^2<+\infty \ \}\ ,$$
where $\Pi $ denotes the Szeg\"o projector from $L^2$ to $L^2_+$,
$$\forall (c_k)\in \ell ^2(\Z )\ ,\ \Pi (\sum _{k=-\infty }^\infty c_k\, {\rm e}^{ik\theta }\ )=\sum _{k=0}^\infty c_k\, {\rm e}^{ik\theta }\ .$$
If $L^2_+$ is endowed with the symplectic form 
$$\omega (u,v)=4\,{\rm Im}(u|v)\ ,\ (u|v):=\int _{\S^1}u\overline v\, \frac{d\theta }{2\pi }\ ,$$
this system is formally Hamiltonian, associated to the --- densely defined--- energy
 $$E(u)=\int _{\S^1}|u|^4\, \frac{d\theta}{2\pi }\ .$$
The study of this equation as a toy model  of a totally non dispersive Hamiltonian equation is 
motivated in the introduction of \cite{GG}, to which we refer for more detail. 
 In \cite{GG}, we proved that the Cauchy problem for (\ref{szego}) is well-posed in the Sobolev spaces
 $$H^s_+=H^s\cap L^2_+$$
 for all $s\ge \frac 12$. The unexpected feature of this equation is the existence of a Lax pair, in the spirit of Lax \cite{L} for the Korteweg-de Vries equation, and of Zakharov-Shabat \cite{Z}  for the one-dimensional cubic nonlinear Schr\"odinger equation.
Let us describe this structure.  For every $u\in H^{1/2}_+$, we define (see {\it e.g.} Peller
\cite{P}, Nikolskii \cite{N}), the Hankel operator of symbol $u$ by
$$H_u(h)=\Pi (u\overline h)\ ,\ h\in L^2_+\ .$$
It is easy to check that $H_u$ is a $\C $ -antilinear  Hilbert-Schmidt operator and satisfies the following 
symmetry condition,
$$ (H_u(h_1)\vert h_2)=(H_u(h_2)\vert h_1)\ ,\ h_1,h_2\in L^2_+\ .$$
 In \cite{GG}, we proved that there exists  a mapping $u\mapsto B_u$, valued into $\C $-linear skew--symmetric
operators on $L^2_+$, such that $u$ is a solution of (\ref{szego}) if
and only if
\begin{equation}\label{laxpair}
\frac d{dt} H_u=[B_u,H_u]\ .
\end{equation}
An important consequence of this structure  is that, if $u$ is a solution of (\ref{szego}),
then $H_{u(t)}$ is unitarily equivalent to $H_{u(0)}$. In particular, the spectrum of 
the $\C $-linear positive self adjoint trace class operator $H_u^2$ is conserved by the evolution.
Moreover, one can prove that 
\begin{equation}\label{Bu}
B_u=-iT_{\vert u\vert ^2}+\frac i2H_u^2\ ,
\end{equation}
where $T_b$ denotes the Toeplitz operator of symbol $b$,
$$T_b(h)=\Pi (bh)\ .$$
This special form of $B_u$ induces another consequence, namely that, for every Borel function
$f$ bounded on the spectrum of $H_u^2$, the quantity
\begin{equation}\label{jdef}
J[f](u):=(f(H_u^2)(1)\vert 1)
\end{equation}
is a conservation law. Here $f(H_u^2)$ is the bounded operator provided by the spectral theorem.
Let us mention some particular cases of such conservation laws which are of special interest.
If $\lambda ^2$ is an eigenvalue of $H_u^2$ , denote by $P$ the orthogonal projector 
onto the corresponding eigenspace of $H_u^2$. Then
$$\Vert P(1)\Vert ^2=J[{\bf 1}_{\{ \lambda ^2\} }](u)\ .$$
A special role is also played by 
\begin{equation}\label{j2n}
J_{2n}(u)=(H_u^{2n}(1)\vert 1)\ ,\ n\in \Z _+\ ,
\end{equation}
for which $f(s)=s^n$, and by their generating function
\begin{equation}\label{jdex}
J(x)(u)=1+\sum _{n=1}^\infty x^nJ_{2n}(u)=((I-xH_u^2)^{-1}(1)\vert 1)\ .
\end{equation}
for which $f(s)=(1-xs)^{-1}$.
Notice that $E=2J_4-J_2^2$.

A third consequence of the Lax pair structure is the existence of  finite dimensional submanifolds of
$L^2_+$ which are invariant by the flow of (\ref{szego}). By a theorem due to Kronecker \cite{Kr},
the Hankel operator $H_u$ is of finite rank $N$ if and only if $u$ is a rational function 
of the complex variable $z$, with no poles in the unit disc, and of the following form,
$$u(z)=\frac {A(z)}{B(z)}\ ,$$
with $A\in \C_{N-1} [z] $, $B\in \C_N [z],$ $B(0)=1,$ $d(A)=N-1$ or
$d(B)= N$, $A$ and $B$ have no common factors,  and $B(z)\ne 0$ if
$|z|\leq 1$.
Here  $\C _D[z]$ denotes the class of complex polynomials of
degree at most $D$, and $d(A)$ denotes the degree of a polynomial
$A$. We denote by  ${\mathcal M}(N)$ the set of such functions $u$.
It is elementary to check that ${\mathcal M}(N)$ is a
$2N$-dimensional complex submanifold of $L^2_+$. 
In \cite{GG}, we proved that the functions $J_{2n}, n=1,\cdots ,2N$, are 
in involution on ${\mathcal M}(N)$, that their differentials are linearly independent
outside a closed subset of measure $0$, and that the level sets of $(J_1,\dots, J_{2N})$ are generically compact in 
${\mathcal M}(N)$. By the Liouville-Arnold theorem \cite{Ar}, the connected components 
of these generic level sets are Lagrangian tori, which are invariant by the flow of
(\ref{szego}).
\s The purpose of this paper is to study  these invariant tori in detail
by introducing the corresponding action-angle variables. As a consequence, this will provide explicit formulae for the 
resolution of the Cauchy problem for (\ref{szego}). Notice that similar coordinates were introduced  for the Korteweg-de Vries equation by Kappeler-P\"oschel \cite{KP} ,
and more recently by Kuksin-Perelman \cite{KuPe} as an application of Vey's theorem,
and,  for the cubic one-dimensional nonlinear Schr\"odinger equation, by Gr\'ebert-Kappeler--P\"oschel \cite{GKP}. 
Our method here is however completely different, since it is based on specific properties of Hankel operators.
We now describe the results in more detail.\subsection{Action angle variables in the finite rank case}
We denote by $ {\mathcal M}(N)_{\rm gen}$ the set of $u\in {\mathcal M}(N)$ such that $1$ does not belong to the range of $H_u$, and such that
the vectors $H_u^{2k}(1)$, $k=1,\dots ,N$, are linearly independent. We proved in \cite{GG}, Theorem 7.1,  that $ {\mathcal M}(N)_{\rm gen}$ is an open subset of $ {\mathcal M}(N)$,
whose complement is of Lebesgue measure $0$. Moreover, it can be shown that $ {\mathcal M}(N)_{\rm gen}$ is the set of $u$ such that $H_u^2$ admits exactly $N$ simple positive eigenvalues $\lambda_1^2>\dots >\lambda_N^2$ with the following additional property,
\begin{equation}\label{genericite}
\nu_j> 0\ {\rm for}\  j=1,\dots,N\ {\rm and}\ \sum _{j=1}^N \nu_j^2<1\ ,
\end{equation}
where, 
for each $j$,  we define the normalization constants
\begin{equation}\label{nuj}
\nu_j:=\Vert P_j(1)\Vert\ ,
\end{equation} 
and where $P_j$ denotes  the orthogonal projector onto the eigenspace $E_j$ of $H_u^2$ associated to $\lambda_j^2$.
Indeed, given an orthonormal basis $(e_1,\dots ,e_N)$ of the range of $H_u$ such that $H_u^2e_j=\lambda _j^2e_j$,  the modulus of the determinant  of the vectors $H_u^{2k}(1)$, $k=1,\dots ,N$
in this basis  is equal to
$$\vert (1\vert e_1)\vert \dots \vert (1\vert e_N)\vert \, \vert \det (\lambda _j^{2k})_{1\leq j,k\leq N}\vert \ .$$
Moreover, $\sum _j\nu _j^2$ is the square of the norm of the orthogonal projection of $1$ onto the range of $H_u$, hence is $<1$ if and only if $1$ does not belong to the range of $H_u$.
\s We then define our new variables.
The first set of action variables is given by $$I_j(u)=2\lambda_j^2\ ,\  j=1,\dots ,N.$$ We define the first set of angle variables as follows.
Using the antilinearity of $H_u$ there exists an orthonormal basis $(e_1, \cdots ,e_N)$ of the range of $H_u$ such that
$$H_u(e_j)=\lambda _je_j\ ,\  j=1,\cdots ,N.$$
Notice that  the orthonormal system $(e_1,\cdots ,e_N)$ is determined by $u$ up to a change of sign on 
some of the $e_j$, in other words up to the action of $\{  \pm 1\} ^N$ acting as a group of isometries. 
Therefore we can define the angles
$$ \varphi_j(u):={\rm arg}(1|e_j)^2\; j=1,\dots,N\ .$$
Since $\mathcal M(N)_{{\rm gen}}$ is a symplectic manifold of real dimension $4N$, it remains to define $N$ other action variables and $N$ other angle variables. We do the same analysis with the operator $K_u=H_u T_z$ as the one we did with $H_u$. Here $T_z$ is nothing but the multiplication by $z$, namely the shift operator on the Fourier coefficients. We will show that $K_u^2$, which turns out to be a self-adjoint positive operator,  has $N$ distinct eigenvalues denoted by $\mu^2_1>\mu^2_2>\dots>\mu^2_N$.
Furthermore, the $\mu_j^2$ are the $N$ solutions of the equation in $\sigma$,
\begin{equation}\label{eqmu}
\sum_{j=1}^N\frac{\lambda_j^2\nu_j^2}{\lambda_j^2-\sigma}=1
\end{equation}
satisfying 
\begin{equation}\label{lambdamu}
\lambda_1^2>\mu^2_1>\lambda_2^2>\mu^2_2>\dots>\lambda^2_N>\mu^2_N>0.
\end{equation}
As before, by the antilinearity of $K_u$ there exists an orthonormal basis $(f_1, \cdots ,f_N)$ of the range of $K_u$ such that
$$K_u(f_m)=\mu _mf_m\ ,\  m=1,\cdots ,N,$$
and $(f_1,\cdots ,f_N)$ is determined by $u$ up to a change of sign on 
some of the $f_m$. 
We set $$L_m(u):=2\mu_m^2,\;j=1,\dots,N\text{ and }\theta_m(u):={\rm arg}(u|f_m)^2,\; m=1,\dots,N\ .$$
Define 
\begin{eqnarray*}
\Omega _N:=\{(I_1,\dots,I_N,L_1,\dots,L_N)\in \R^{2N};\; 
I_1>L_1>I_2>\dots>I_N>L_N>0\}\ .\end{eqnarray*} 
Our main result reads
\begin{theorem}\label{TheoDiffeo}
The mapping   
$$\chi_N:=(I_1,\dots,I_N,L_1,\dots,L_N;\varphi_1,\dots,\varphi_N,\theta_1,\dots,\theta_N)$$  is a symplectic diffeomorphism
from $\mathcal M(N)_{\rm gen}$ onto $\Omega _N\times \T^{2N}$, in the sense that 
\begin{equation}\label{ChiStarOmega}
\chi_{N*}\omega=\sum_{j=1}^NdI_j\wedge d\varphi_j+\sum_{m=1}^NdL_m\wedge d\theta_m
\end{equation}
\end{theorem}
As we will see in the proof, a complement to this theorem is an explicit formula giving $u$ in terms
of $\chi _N(u)$--- see Proposition \ref{inversespectral} below. As a first consequence of this result, we obtain an explicit solution to the Cauchy problem for (\ref{szego})
for data in $\mathcal M(N)_{{\rm gen}}$. 
\begin{corollary}\label{szegosolution}
The cubic Szeg\"o equation (\ref{szego}) is equivalent, in the above  variables, to the system
 $$\left\{\begin{array}{cc}
\dot I_j=0, \; \dot L_m=0\\
\dot\varphi_j=\frac 12 I_j\ ,\ \dot \theta_m=-\frac 12 L_m\end{array}\right.$$
\end{corollary}
\subsection{The infinite dimensional case}
 Theorem \ref{TheoDiffeo} and Corollary \ref{szegosolution}  admit natural generalizations to infinite dimension.
In this case, we define the set $H^{1/2}_{+,\rm{gen}}$ as the subset of functions $u$ in $H^{1/2}_+$ so that
 $H_u^2$ admits only simple positive eigenvalues
$$\lambda_1^2>\lambda_2^2>\dots$$ on the closure of its range,  
and such that, for any $j\ge 1$, $$\nu_j:=\Vert P_j(1)\Vert \neq 0.$$
We  shall prove that $H^{1/2}_{+,\rm{gen}}$ is a dense $G_\delta $ set in $H^{1/2}_+$ and that the motion stays on infinite dimensional invariant tori, leading to almost periodic solutions valued in $H^{1/2}_+$. 
More precisely, denoting by $(\mu _m^2)_{m\ge 1}$ the sequence of positive eigenvalues of $K_u^2$, and observing that 
 $${\lambda _1^2>\mu^2 _1>\lambda _2^2>\mu^2 _2>\dots },$$
we can define as before orthonormal systems $(e_j)_{j\ge 1}$ and $(f_m)_{m\ge 1}$, with
$$H_u(e_j)=\lambda _je_j\ ,\ K_u(f_m)=\mu _m f_m\ .$$
As before, we introduce the following sequences of angles,
$$\varphi _j=\arg (1\vert e_j)^2\ ,\ \theta _m=\arg (u\vert f_m)^2\ ,\ j,m\ge 1\ .$$
We then have the following generalization of Theorem \ref{TheoDiffeo} and of Corollary \ref{szegosolution}.
 \begin{theorem}\label{TheoHomeo}
  The mapping
 $${\chi : u\in H^{\frac 12}_{+,{\rm gen}}\mapsto ((\zeta _j:=\lambda _j\,{\rm e}^{-i\varphi _j})_{j\ge 1}\ ,\ (\gamma _m:={\mu _m}\,{\rm e}^{-i\theta _m})_{m\ge 1})}$$
 is a homeomorphism onto the subset of ${\ell ^2\times \ell ^2}$ defined by
 $$\Xi:=\left\{((\zeta_j)_{j\ge 1},(\gamma_m)_{m\ge 1})\in\ell^2 \times \ell ^2,\; \vert \zeta _1\vert >\vert \gamma _1\vert >\vert \zeta _2\vert >\vert \gamma _2\vert >\dots >0 \right\}.$$
 Moreover, the evolution of (\ref{szego}) reads through $\chi $ as
 $${i\dot \zeta _j=\lambda _j^2\zeta _j\ ,\ i\dot \gamma _m =-\mu^2 _m\gamma _m \ .}$$
\end{theorem} 
This theorem  is deduced from Theorem \ref{TheoDiffeo} through an approximation argument by the finite rank case. The convergence of this approximation is  a consequence of a compactness result 
on families of Hankel operators--- see Proposition \ref{compactness} below.

\subsection{Application to inverse problems for Hankel operators}
Theorems \ref{TheoDiffeo} and \ref{TheoHomeo} can be rephrased as solutions  to inverse spectral problems on Hankel operators. We denote by ${\bf h}^{1/2}$ the space of  sequences $(c_n)_{n\ge 0}$ of complex numbers such that
\begin{equation}\label{hundemi}
\sum _{n=0}^\infty n\vert c_n\vert ^2<\infty \ ,
\end{equation}
endowed with its natural norm.
Given $c\in {\bf h}^{1/2}$, we define the operator $\Gamma _c :\ell ^2(\N)\rightarrow \ell^2(\N)$ by
$$\forall x=(x_n)_{n\ge 0}\in \ell ^2(\N)\ ,\ \Gamma _c(x)_n=\sum _{p=0}^\infty c_{n+p}x_p\ .$$
In view of (\ref{hundemi}), it is clear that $\Gamma _c$ is Hilbert--Schmidt. We also introduce
$$\tilde \Gamma _c:=\Gamma _{\tilde c}$$
where 
$$\forall n\in \N\ ,\ \tilde c_n:=c_{n+1}\ .$$
Our first result concerns the prescription of positive singular values of both $\Gamma _c$ and $\tilde \Gamma _c$. Recall that the positive singular values of an operator $A$ are the positive eigenvalues 
of the operator $\sqrt{AA^*}$.
\begin{corollary}\label{inverseHankel} 

 -- Let $(\lambda_j)_{1\le j\le N}$,  $(\mu_j)_{1\le j\le N}$ be  $N$-tuples of  real numbers satisfying
$$\lambda_1>\mu _1>\lambda_2>\mu _2>\dots>\lambda_N>\mu _N>0\ .$$
The set of  sequences $c\in {\bf h}^{1/2}$ such that  $\Gamma _c$ has rank $N$ and admits   $\lambda_j\ ,\ 1\le j\le N,$ as simple positive singular values,
 and such that  $\tilde \Gamma _c$ has rank $N$ and admits  $\mu _j,\ 1\le j\le N,$ as simple positive singular values, is  a Lagrangian torus of dimension $2N$.
 
-- Let $(\lambda_j)_{ j\ge 1}$,  $(\mu_m)_{m\ge 1}$ be  sequences  of positive real numbers satisfying
$$\lambda_1>\mu _1>\lambda _2>\mu _2>\dots >0 \ ,\ \sum _{j=1}^\infty \lambda _j^2<\infty \ .$$
The set of  functions $c\in {\bf h}^{1/2}$ such that the positive singular values of $\Gamma _c$ are $\lambda _j,\  j\ge 1,$ and are simple,
and such that the positive singular values  of   $\tilde \Gamma _c$
are  $\mu _m,\,  m\ge 1,$ and are simple,  is  an infinite dimensional  torus.
\end{corollary}
In the particular case of real values sequences $c$ in ${\bf h}^{1/2}$, $\Gamma _c$ is self-adjoint and Corollary \ref{inverseHankel} has the following simple reformulation.
\begin{corollary}\label{Hankelauto}

-- Let $\zeta _1\dots ,\zeta  _N,\gamma_1,\dots ,\gamma_N$ be  real numbers such that
$$\vert \zeta _1\vert >\vert \gamma_1\vert >\vert \zeta  _2\vert >\vert \gamma_2\vert >...>\vert \zeta _N\vert >\vert \gamma_N\vert >0\ .$$
There exists a unique sequence $c=(c_n)_{n\ge 0}$ of real numbers such that 
$\Gamma _c$ 
has rank $N$ with non zero eigenvalues $\zeta _1,\dots ,\zeta  _N$, and such that the selfadjoint operator
$\tilde \Gamma _c$
has rank $N$ with non zero eigenvalues $\gamma_1,\dots ,\gamma_N$.

-- Let $( \zeta _j)_{j\ge 1}$ , $( \gamma _m)_{m\ge 1}$ be two sequences of real numbers such that
$$\vert  \zeta _1\vert >\vert  \gamma_1\vert >\vert \zeta  _2\vert >\vert \gamma_2\vert >... >0\ , \sum _{j=1}^\infty \zeta _j^2<\infty \ .$$
There exists a unique sequence $c\in {\bf h}^{1/2}$ of real numbers such that
  the non zero eigenvalues of the selfadjoint operator
$\Gamma _c$
are   $\ \zeta _j,\  j\ge 1,$ and are simple, 
and the non zero eigenvalues of the selfadjoint operator
$\tilde \Gamma _c$
are   $ \gamma_m,\  m\ge 1,$ and are simple.
\end{corollary}
Notice that, in \cite{T1} and \cite{T2},  Treil proved that any noninvertible nonnegative operator on a Hilbert space, with simple discrete spectrum, and  which is either one to one or
with infinite dimensional kernel, is unitarily equivalent to the modulus of a Hankel operator. This implies in particular that any decreasing sequence of positive numbers in $\ell ^2$  is the sequence of 
the positive singular values of a Hilbert-Schmidt Hankel operator. In Corollary \ref{inverseHankel}, we prove that it is possible
to prescribe both singular values of $\Gamma _c$ and of $\tilde \Gamma _c$, assuming that  they are all simple and distinct, and we describe the set of solutions as a torus.

As for Corollary \ref{Hankelauto}, it has to be compared to the result by Megretskii, Peller, Treil \cite{MPT}, who characterized in the widest generality
the self-adjoint operators which are unitarily equivalent to Hankel operators. In the special case of  Hilbert-Schmidt operators with simple non zero eigenvalues,
Corollary \ref{Hankelauto} establishes  that it is possible to impose the spectrum of both $\Gamma _c$ and $\tilde \Gamma _c$, and that this completely  characterizes the symbol.

Finally, let us emphasize that Corollaries \ref{inverseHankel} and \ref{Hankelauto} are completed by an explicit formula which gives the sequences $c$, see Remark \ref{cn} below.

\subsection{Stability of invariant tori and instability of traveling waves}
Given $(I_1,\dots,I_N,L_1,\dots,L_N)\in \Omega _N$, denote by ${\bf T}(I_1,\dots,I_N,L_1,\dots,L_N)$ the corresponding Lagrangian torus in ${\mathcal M}(N)_{{\rm gen}}$ via $\chi _N$.
Our  next result is a variational characterization of ${\bf T}(I_1,\dots,I_N,L_1,\dots,L_N)$ which implies its stability through the evolution of the cubic Szeg\"o equation, analogously to the result by Lax \cite{L2} for 
KdV. We recall the notation
$$\forall u\in H^{1/2}_+, M(u)=(-i\partial _\theta u\vert u)=\sum _{k=0}^\infty k\vert \hat u(k)\vert ^2\ .$$
\begin{theorem}\label{stability}
For $n=1,\dots ,2N$, define
\begin{equation}
j_{2n}=\sum _{j=1}^N2^{-n}I_j^n\left (1-\frac {L_j}{I_j}\right )\prod _{k\ne j}\left (\frac{L_k-I_j}{I_k-I_j}\right )\ .
\end{equation}
Then ${\bf T}(I_1,\dots,I_N,L_1,\dots,L_N)$ is the set of the solutions in $H^{1/2}_+$ of the minimization problem
$$\inf \{ M(u)\;  : \; J_{2n}(u)=j_{2n}\ ,\ n=1,\dots ,2N\} \  .$$
Consequently, ${\bf T}:={\bf T}(I_1,\dots,I_N,L_1,\dots,L_N)$ is stable under the evolution of (\ref{szego}), in the sense that,
for every $\varepsilon >0$, there exists $\delta >0$ such that, if 
$$\inf _{v\in{\bf T}}\Vert u_0-v\Vert _{H^{1/2}}\le \delta \ ,$$
then the solution $u$ of (\ref{szego}) with $u(0)=u_0$ satisfies 
$$\sup _{t\in \R }\inf _{v\in{\bf T}}\Vert u(t)-v\Vert _{H^{1/2}}\le \varepsilon \  .$$
\end{theorem}
Let us mention that there is a similar result for the infinite dimensional tori deduced from 
Theorem \ref{TheoHomeo} --- see Remark \ref{stabilityinfini} below.
\s
Our next observation concerns the case $N=1$, where ${\bf T}(I_1,L_1)$ consists exactly of functions 
\begin{equation}\label{soliton}
u_{\alpha ,p}(z)=\frac \alpha {1-pz}
\end{equation}
where $\vert \alpha \vert $ and $\vert p\vert $ are fixed positive numbers which depend on $I_1,L_1$. 
In \cite{GG}, it was observed that such functions $u$ are traveling waves of equation (\ref{szego}), in the sense 
that there exists $(\omega ,c)\in \R ^2$ such that 
$$t\mapsto {\rm e}^{-i\omega t}u_{\alpha ,p}(z{\rm e}^{-ict})$$
is a solution to (\ref{szego}). Moreover,  Proposition 5 and Corollary 4 of \cite{GG} establish the orbital stability of this traveling wave as the solution
of a variational problem, which is exactly 
the statement of Theorem \ref{stability} in this case. Therefore it is natural to address the question of orbital stability for all the traveling waves of (\ref{szego}), which were classified in
Theorem 1.4 of \cite{GG}. The next result gives a complete answer to this question.
\begin{theorem}
If $u$ is a traveling wave of (\ref{szego}) which is not of the form $u_{\alpha ,p}$ 
as defined in (\ref{soliton}), then $u$ is orbitally unstable.
\end{theorem}
The proof of this theorem is based on the explicit resolution of Equation (\ref{szego}) when the Cauchy data are suitable perturbations of the traveling wave $u$.
 \subsection{Organization of the paper}
We close this introduction by describing the organization of the paper. In Section \ref{preli}, we introduce some fundamental tools which will be used in the paper, including the compressed shift operator, 
a trace formula and a compactness result. In Section \ref{actionangle}, we prove Theorem \ref{TheoDiffeo} on action-angle variables in the finite rank case and its corollary
about the explicit solution of (\ref{szego}). Section \ref{infini} contains the  generalization to infinite dimension stated in Theorem \ref{TheoHomeo}. Section \ref{invhankel} is devoted to the solution of  inverse spectral problems for Hankel operators  as stated in Corollaries \ref{inverseHankel} and \ref{Hankelauto}. In Section \ref{stabilitysection}, we prove Theorem \ref{stability} about stability of invariant tori.
Finally, Section \ref{instability} establishes the orbital instability of traveling waves.
\end{section}
\begin{section}{Preliminaries}\label{preli}
\subsection{The compressed shift operator}
We are going to use the well known link between the shift operator and the Hankel operators. Namely, if $T_z$ denotes the shift operator --- the Toeplitz operator of symbol $z\mapsto z$ ---, one can easily check the following identity,
\begin{equation}\label{Shift}H_uT_z=T_z^*H_u.\end{equation}
With the notation introduced in the introduction, it reads
$$K_u=T_z^*H_u.$$
Moreover,
$$K_u^2=H_uT_zT_z^*H_u=H_u(I-(\, .\,  \vert 1))H_u=H_u^2-(\, .\, \vert u)u\ .$$
We introduce the compressed shift operator  (\cite{N}, \cite{N2}, \cite{P})
$$S:=P_uT_z\ ,$$
 where $P_u$ denotes the orthogonal projector onto the closure of the range of $H_u$. 
By property (\ref{Shift}), $\ker H_u=\ker P_u$ is stable by $T_z$, hence 
 $$S=P_uT_zP_u$$ so that $S$ is an operator from the closure of the range of $H_u$ into itself. In the sequel, we shall always denote by $S$ the induced operator
 on the closure of the range of $H_u$, and  by $S^*$ the adjoint of this operator. 
 \s
Now observe that  operator $S$ arises in the Fourier series decomposition of $u$, namely
\begin{equation}\label{FourierU}
u(z)=
\sum_{n=0}^\infty (u|z^n)z^n=\sum_{n=0}^\infty (u|T^n_z(1))z^n=\sum_{n=0}^\infty (u|S^nP_u(1))z^n.
\end{equation}
As a consequence, we have, for $|z|<1$,
\begin{equation}\label{InverseSpectral}
u(z)=(u|(I-\overline zS)^{-1}P_u(1)).
\end{equation}
which makes sense since $\Vert S\Vert \le 1$. In the next sections, we shall see how the above formula leads to an inverse formula for the maps $\chi _N$ and $\chi $.
\subsection{A trace formula and a compactness result}
For every integer $j\ge 1$, we denote by $\mathcal F _j$ the set of subspaces of $L^2_+$ of dimension at most $j$.
Given $u\in H^{1/2}_+$, we define $\lambda _j(u)\ge 0$ by 
$$\lambda _j^2(u)=\min _{F\in \mathcal F_{j-1}}\max _{h\in F^\perp, \Vert h\Vert =1}(H_u^2(h)\vert h)\ .$$
The following  is a standard fact about nonnegative compact operators.
\begin{itemize}
\item If $H_u^2$ has finite rank $N$, then $\lambda _j^2(u)=0$ for every $j>N$, and $\lambda _1^2(u)\ge \lambda _2^2(u)\ge \dots \ge \lambda _n^2(u)>0$ are the positive eigenvalues
of $H_u^2$, listed according to their multiplicities.
\item If $H_u^2$ has infinite rank, then $\lambda _1^2(u)\ge \lambda _2^2(u)\ge \dots >0$ are the positive eigenvalues
of $H_u^2$, listed according to their multiplicities.
\end{itemize}
Likewise, we define $\mu _j(u)\ge 0$ by 
$$ \mu _j^2(h)=\min _{F\in \mathcal F_{j-1}}\max _{h\in F^\perp, \Vert h\Vert =1}(K_u^2(h)\vert h)=\min _{F\in \mathcal F_{j-1}}\max _{h\in F^\perp, \Vert h\Vert =1}(H_u^2(h)\vert h)-\vert (h\vert u)\vert ^2\ .$$
From these formulae, it is easy to check that
$$\lambda _1(u)\ge \mu _1(u)\ge \lambda _2(u)\ge \mu _2(u)\ge \dots $$
The following result makes an important connection with function $J(x)$ introduced in (\ref{jdex}).
 \begin{proposition}\label{trace}
For every $u\in H
 ^{1/2}_+$, the following identities hold.
 \begin{equation}\label{traceformula}
  \sum _{j=1}^\infty\left ( \frac{\lambda _j^2}{1-\lambda _j^2x}-\frac{\mu _j^2}{1-\mu _j^2x}\right )=\frac{J'(x)}{J(x)}\ ,\ x\notin \left \{ \frac 1{\lambda _j^2}, \frac 1{\mu _j^2}, j\ge 1\right \}\ .
  \end{equation}
 \begin{equation}\label{ProdHmu}
 J(x)=\prod_{j=1}^\infty \frac{1-\mu^2_jx}{1-\lambda_j^2x}\ ,\  x\notin \left \{ \frac 1{\lambda _j^2}, j\ge 1\right \}.
 \end{equation}
\end{proposition}
 \begin{proof}
 First notice that (\ref{ProdHmu}) is a direct consequence of (\ref{traceformula}) by integration and the fact that $J(0)=1$.
 It remains to prove (\ref{traceformula}), which we shall interpret as a trace formula.
Indeed, recall that
$$K_u^2(h)=H_u^2(h)-(h\vert u)u,$$
so that an elementary calculation yields
$$(I-xH_u^2)^{-1}(f)-(I-xK_u^2)^{-1}(f)=\frac{x}{J(x)}(f\vert (I-xH_u^2)^{-1}(u))(I-xH_u^2)^{-1}(u)\ .$$
Consequently,
$${\rm Tr}((I-xH_u^2)^{-1}-(I-xK_u^2)^{-1})=\frac{x}{J(x)}\Vert (I-xH_u^2)^{-1}(u)\Vert ^2\ .$$
Since, on the one hand,
$$\Vert (I-xH_u^2)^{-1}(u)\Vert ^2=((I-xH_u^2)^{-1}H_u^2(1)\vert 1)=\frac d{dx}((I-xH_u^2)^{-1}(1)\vert 1)=J'(x)\ ,$$
and on the other hand
\begin{eqnarray*}
{\rm Tr}((I-xH_u^2)^{-1}-(I-xK_u^2)^{-1})&=&x{\rm Tr} \left (H_u^2(I-xH_u^2)^{-1}-K_u^2(I-xK_u^2)^{-1}\right )\\
&=&x \sum _{j=1}^\infty\left ( \frac{\lambda _j^2}{1-\lambda _j^2x}-\frac{\mu _j^2}{1-\mu _j^2x}\right )\ ,
\end{eqnarray*}
Formula (\ref{traceformula}) follows.
 \end{proof}
 From the above proposition, we infer the following compactness result, which will be of constant use throughout the paper.
 \begin{proposition}\label{compactness}
 Let $(u_p)$ be a sequence of $H^{1/2}_+$ weakly convergent to $u$ in $H^{1/2}_+$. We assume that 
 $$(\lambda _j(u_p))_{j\ge 1}\td_p,\infty  (\overline \lambda _j)_{j\ge 1}\ ,\ (\mu _j(u_p))_{j\ge 1}\td_p,\infty  (\overline \mu _j)_{j\ge 1}\  {\rm in}\  \ell ^2,$$
 and the following simplicity assumptions :
 \begin{itemize}
\item If $j>k$ and $\overline\lambda _j>0$,  then $\overline\lambda _j> \overline\lambda _k$.
\item If $j>k$ and $\overline\mu _j>0$, then $\overline\mu _j>\overline\mu _k$.
\item If $\overline\lambda _j>0$ for some $j\ge 1$, then $\overline\lambda _j\ne \overline\mu _m$ for every $m\ge 1$. 
 \end{itemize}
 Then, for every $j\ge 1$, $\lambda _j(u)=\overline\lambda _j$, $\mu _j(u)=\overline\mu _j$, and the convergence of $u_p$ to $u$ is strong in $H^{1/2}_+$.
 \end{proposition}
\begin{proof}
Firstly, we make a connection between the sequences $(\overline \lambda _j)_{j\ge 1},( \overline \mu _j)_{j\ge 1}$ and $(\lambda _j(u))_{j\ge 1},(\mu _j(u))_{j\ge 1}$
by means of standard functional analysis.
\begin{Lemma}\label{EigenvalueLimit}
Let $(A_p)$ be a sequence of compact selfadjoint nonnegative operators on a Hilbert space $\mathcal H$, which strongly converges  to $A$, namely
$$\forall h\in \mathcal H\ ,\ A_ph\td_p,\infty Ah\ .$$
For every $j\ge 1$, denote by $\mathcal F _j$ the set of subspaces of $\mathcal H$ of dimension at most $j$,  set
$$a_j^{(p)}=\min _{F\in \mathcal F_{j-1}}\max _{h\in F^\perp, \Vert h\Vert =1}(A_p(h)\vert h)\ ,$$
and assume
 $$a_j^{(p)} \rightarrow \overline a_j$$
 with, if $j>k$ and $\overline a_j\ne 0$, $\overline a_j> \overline a_k$. 
  Then the positive eigenvalues of $A$ are simple and belong to the limit set  $\{ \overline a_j\}$. 
\end{Lemma}  
\begin{proof}
Denote by $(e_j^{(p)})$ an orthonormal basis of $\ker A_p^\perp $ with $A_pe_j^{(p)}=a_j^{(p)}e_j^{(p)}$.
For every 
$h\in \mathcal H$, we decompose
$$h=\sum _j(h\vert e_j^{(p)})e_j^{(p)}+h_0^{(p)}$$
where $h_0^{(p)}\in \ker A_p$.   Let $a\in \R $. Then, passing to the limit in
\begin{equation}\label{decomposition}
\Vert (A_p-a)h\Vert ^2=\sum _j(a_j^{(p)}-a)^2\vert (h\vert e_j^{(p)})\vert ^2+a^2\Vert h_0^{(p)}\Vert ^2\ ,
\end{equation}
we get
$$\Vert (A-a)h\Vert \ge \min (\inf _j\vert \overline a_j-a\vert ,\vert a\vert )\Vert h\Vert $$
and therefore, if  $a \notin \{\overline a_j\}\cup \{ 0\} $, $a$ is not an eigenvalue of $A$. 
Assume now that $a=\overline a_j$, and come back to (\ref{decomposition}). If $Ah=\overline a_jh$, we infer
$$\sum _{k\ne j}\vert (h\vert e_k^{(p)})\vert ^2+\Vert h_0^{(p)}\Vert ^2\rightarrow 0$$
or $\Vert h-(h\vert e_j^{(p)})e_j^{(p)}\Vert ^2\rightarrow 0$. Consequently, given  eigenvectors $h_1, h_2$  
of $A$ with eigenvalue $\overline a_j$, we have
$$\vert (h_1\vert h_2)\vert =\lim \vert (h_1\vert e_k^{(p)})\vert \, \vert (h_2\vert e_k^{(p)})\vert=\Vert h_1\Vert \, \Vert h_2\Vert \ ,$$
which means that $\overline a_j$ is a simple eigenvalue.
\end{proof}
Let us return to the proof of Proposition \ref{compactness}. By the Rellich theorem, $u_p$ tends to $u$ strongly in $L^2_+$, hence, for every $h\in L^2_+$, we have 
\begin{equation}\label{strongHu}
H_{u_p}(h)\td _p,\infty H_u(h)\ .
\end{equation}
Since the norm of $H_{u_p}$ is bounded by its Hilbert-Schmidt norm, namely the $H^{1/2}$ norm of $u_p$, which is bounded,
we conclude that (\ref{strongHu}) holds uniformly for $h$ in every compact subset of $L^2_+$, hence 
$$\forall n\ge 1, H_{u_p}^n(h)\td _p,\infty H_u^n(h)\ .$$
In particular, for every $n\ge 1$,
$$J_{2n}(u_p):=(H_{u_p}^{2n}(1)\vert 1)\td _p,\infty  (H_u^{2n}(1)\vert 1):=J_{2n}(u)\ ,$$
and there exists $C>0$ such that
$$\forall n\ge 1, \sup _p J_{2n}(u_p)\le C^n\ .$$
Choose $\delta >0$ such that $\delta C<1$. Then, for every real number $x$ such that $\vert x \vert <\delta $, we have, by dominated convergence,
$$J(x)(u_p):=1+\sum _{n=1}^\infty x^nJ_{2n}(u_p)\td _p,\infty 1+\sum _{n=1}^\infty x^nJ_{2n}(u):=J(x)(u)>0\ .$$
Similarly,
$$J'(x)(u_p)\td _p,\infty J'(x)(u)\ ,$$
and therefore 
$$\frac{J'(x)(u_p)}{J(x)(u_p)}\td _p,\infty \frac{J'(x)(u)}{J(x)(u)}\ .$$
On the other hand, in view of the assumption about $\ell ^2$ convergence of $(\lambda _j(u_p))_{j\ge 1}$ and $(\mu _j(u_p))_{j\ge 1}$, we also have, for $\vert x\vert < \delta $,
$$  \sum _{j=1}^\infty\left ( \frac{\lambda _j^2(u_p)}{1-\lambda _j^2(u_p)x}-\frac{\mu _j^2(u_p)}{1-\mu _j^2(u_p)x}\right )\td _p,\infty   \sum _{j=1}^\infty\left ( \frac{\overline \lambda _j^2}{1-\overline\lambda _j^2x}-\frac{\overline\mu _j^2}
{1-\overline\mu _j^2x}\right )\ $$
Using Formula (\ref{traceformula}) of Lemma \ref{trace} above, we infer
\begin{equation}\label{idlimit}
 \sum _{j=1}^\infty\left ( \frac{\overline \lambda _j^2}{1-\overline\lambda _j^2x}-\frac{\overline\mu _j^2}{1-\overline\mu _j^2x}\right )=
\sum _{j=1}^\infty\left ( \frac{\lambda _j^2(u)}{1-\lambda _j^2(u)x}-\frac{\mu _j^2(u)}{1-\mu _j^2(u)x}\right )\ ,
\end{equation}
for $\vert x\vert <\delta$, and hence for every $x$ distinct from the poles, by analytic continuation.
By the assumption of the proposition, no cancellation can occur in the left hand side of (\ref{idlimit}), and the pole are all distinct. On the other hand, applying Lemma \ref{EigenvalueLimit}
to $A_p=H_{u_p}^2$ and to $A_p=K_{u_p}^2$, we know that 
$$\{ \lambda _j^2(u), j\ge 1\} \subset \{ \overline \lambda _j^2, j\ge 1\} \ ,\ \{ \mu _j^2(u), j\ge 1\} \subset \{ \overline \mu _j^2, j\ge 1\}$$
and that the multiplicity of positive eigenvalues is $1$. Consequently, there is no cancellation in the right hand side of (\ref{idlimit}) either, and all the poles are simple.
We conclude that $\lambda _j(u)=\overline\lambda _j$, $\mu _j(u)=\overline\mu _j$ for every $j\ge 1$. Moreover,
$${\rm Tr}(H_u^2)=\lim _{p\rightarrow \infty}{\rm Tr}(H_{u_p}^2), $$
which, since ${\rm Tr}(H_u^2)\simeq \Vert u\Vert _{H^{1/2}}^2$, implies the strong convergence in $H^{1/2}$.
\end{proof}
\end{section}
\begin{section}{The action-angle variables }\label{actionangle}
In this section we prove Theorem \ref{TheoDiffeo} and its corollaries \ref{szegosolution} and \ref{inverseHankel}. The proof of Theorem \ref{TheoDiffeo} is split into five parts. 
Firstly, we  study  the compressed shift operator in connection to the spectral theory of $K_u^2$. As a  second step,
using the compressed shift operator, we prove that the unknown $u$ can be recovered from $\chi _N(u)$, with an explicit formula. The third step is devoted to calculating the Poisson brackets 
between action functions $(I,L)$ and angle functions $(\varphi, \theta )$, which implies in particular that $\chi _N$ is a local diffeomorphism. This calculation is achieved thanks to function $J(x)$, the Hamiltonian flow of which satisfies a Lax pair structure,
as we proved in \cite{GG}. The surjectivity of $\chi_N$ is obtained in the fourth step thanks to a topological argument, while the remaining Poisson brackets are calculated in the fifth step.

\subsection{Spectral theory of $K_u^2$ and the compressed shift operator.}\label{SKu}
As a first step,  for $u\in \mathcal M(N)_{\rm gen}$, we study the eigenvalues of  $K_u^2$
on the range of $H_u$. We first observe that $0$ cannot be an eigenvalue. Indeed, otherwise there would exist $g$ in the range of $H_u$ such that $$K_ug=0=T_{\overline z}H_ug,$$
which means that $H_ug$ is a non zero constant. This would imply that $1$ belongs to the range of $H_u$, which contradicts the definition of $\mathcal M(N)_{\rm gen}$. On the other hand, 
if $g$ is an eigenvector associated to an eigenvalue $\sigma >0$, we have, from the identity $K_u^2=H_u^2-(\, .\, \vert u)u\ ,$
\begin{equation}\label{eigen}
(H_u^2-\sigma I)g=(g\vert u)u\ .
\end{equation}
We first claim that $\sigma $ does not belong to $\{ \lambda _1^2\dots ,\lambda _N^2\} $. Indeed, assume $\sigma =\lambda _j^2$ in (\ref{eigen}). 
If  $(g\vert u)=0$, (\ref{eigen}) implies that $g=ke_j$ for some $k\ne 0$,  therefore $(g\vert u)=k\lambda _j(1\vert e_j)$
and this would contradict the assumption $\nu _j> 0$ --- see (\ref{genericite}). If $(g\vert u)\ne 0$, (\ref{eigen}) implies that $u$ belongs to the range of $H_u^2-\lambda _j^2I$, hence
$u$ is orthogonal to $e_j$, which  again is in contradiction with the assumption $\nu _j> 0$.

Therefore (\ref{eigen}) yields
$$g=(g\vert u)(H_u^2-\sigma I)^{-1}u\ ,$$
which is possible if and only if 
$$((H_u^2-\sigma I)^{-1}u\vert u)=1\ ,$$
or, by decomposing $u$ on the $e_j$'s,
$$\sum _{j=1}^N\frac{\lambda _j^2\nu _j^2}{\lambda _j^2-\sigma }=1\ ,$$
which is exactly (\ref{eqmu}). Notice that, as a function of $\sigma $, the left hand side of the above equation
 increases from $-\infty $ to $+\infty $ on each interval between two successive  $\lambda _j^2$, hence the equation admits 
exactly $N$ solutions $\mu _1^2,\dots ,\mu _m^2$.  Summing up, we have proved  that the eigenvalues of $K_u^2$ on the range of $H_u$ are precisely the $\mu^2_m$, $m=1,\dots, N$,  defined by (\ref{eqmu}), with eigenvectors
\begin{equation}\label{Fm}g_m=(H_u^2-\mu^2_mI)^{-1}(u).\end{equation}
Since these eigenvalues are simple, and since $K_u(g_m)$ is also an eigenvector associated to $\mu _m^2$, we have
$$K_u(g_m)=\gamma _mg_m$$
with $\vert \gamma _m\vert ^2=\mu _m^2$. Then an orthonormal basis $(f_1,\dots, f_N)$ of the range of $H_u$ satisfying 
$$K_u(f_m)=\mu _mf_m\  $$
is given by
$$f_m=\frac{ \gamma _m^{1/2}g_m}{\mu _m^{1/2}\Vert g_m\Vert}\ ,$$
so that, using that 
$$(u\vert g_m)=(u\vert (H_u^2-\mu _m^2I)^{-1}(u))=1\ ,$$
in view of (\ref{eqmu}), we have
$$\theta _m:=\arg (u\vert f_m)^2=\arg (\overline \gamma _m)\ .$$
Finally, we have proved that
$$K_u(g_m)=\mu _m\, {\rm e}^{-i\theta _m}g_m\ .$$
Next we come to the link with operator $S$.  Recalling the expression (\ref{Fm}) of $g_m$ and the fact that $K_u=H_uS$,
we infer, using the injectivity of $H_u$ on the range of $H_u$,
\begin{equation}\label{SFm}
S(g_m)=\mu_m\, {\rm e}^{i\theta _m}h_m\end{equation}
where 
$$h_m:=(H_u^2-\mu^2_mI)^{-1}P_u(1)\ .$$
 We summarize the above result in the following lemma.
\begin{Lemma}\label{spectralS}
The sequence $(g_m)$ defined by (\ref{Fm}) is an orthogonal basis of the range of $H_u$,
on which the compressed shift operator acts as
$$S(g_m)={\mu _m}\, {\rm e}^{i\theta _m}\, h_m\ ,\ h_m:=(H_u^2-\mu^2 _mI)^{-1}P_u(1)\ .$$
\end{Lemma}
\subsection{The inverse spectral formula.}

We now prove that $\chi _N$ is one to one, with an explicit formula describing $u$ in terms of $\chi _N(u)$.

\begin{proposition}\label{inversespectral}
If $\chi _N(u)=(2\lambda_1^2,\dots ,2\lambda _N^2,2\mu^2 _1,\dots ,2\mu^2 _N;\varphi _1,\dots ,\varphi _N,\theta _1,\dots ,\theta _N)$ then
 \begin{equation}\label{ExplicitFormula}
u(z)=X(I-zA)^{-1}Y
\end{equation}
where
\begin{eqnarray*}
X&:=&\left(\lambda_j\nu_j\, {\rm e}^{-i\varphi_j}\right)_{1\le j\le N},\\
Y&:=&\left(\nu_k \right)_{1\le k\le N}^T,
\end{eqnarray*}
 $A:=(A_{j,k})_{1\le j,k\le N}$ is given by 
$$A_{j,k}=\sum_{\ell=1}^N \frac{ \lambda_k\nu_j \nu_k \, {\rm e}^{-i(\varphi_k+\theta_\ell)}}{b_\ell(\lambda_j^2-\mu^2_\ell)(\lambda_k^2-\mu^2_\ell)}{\mu_\ell}\ ,$$
and 
\begin{equation}\label{nulambdamu}
\nu _j:=\left (1-\frac{\mu _j^2}{\lambda _j^2}\right )^{1/2}\prod _{k\ne j}\left (\frac{\lambda _j^2-\mu _k^2}{\lambda _j^2-\lambda _k^2}\right )^{1/2}\ ,
\end{equation}
\begin{equation}\label{Cell}
b_\ell=\sum_{j=1}^N\frac{\lambda_j^2\nu_j^2}{(\lambda_j^2-\mu^2_\ell)^2}=\frac 1{\lambda _\ell ^2-\mu _\ell ^2}\prod _{k\ne \ell}\frac{\mu _\ell ^2-\mu _k^2}{\mu _\ell ^2-\lambda _k^2}\ .
\end{equation}
\end{proposition}

\begin{proof}
Our starting point is the formula (\ref{InverseSpectral}) derived in the last section,
$$u(z)=(u|(I-\overline zS)^{-1}P_u(1))\ ,\ \vert z\vert <1\ .$$
We compute this inner product in the orthonormal basis $(\tilde e_j:= {\rm e}^{i\varphi_j/2}e_j)_{1\le j\le N}$ of the range of $H_u$. 
By definition, we have 
$$P_u(1)=\sum_{1\le m\le N}(1\vert e_j)e_j=\sum_{1\le j\le N}\nu_j\,  {\rm e}^{i\varphi_j/2}e_j=\sum_{1\le j\le N}\nu_j \tilde e_j$$ and
$$u=H_u(P_u(1))=\sum_{1\le j\le N}\lambda_j\nu_j\, {\rm e}^{-i\varphi_j/2}e_j=\sum_{1\le j\le N}\lambda_j\nu_j\, {\rm e}^{-i\varphi_j}\tilde e_j.$$ 
Let us compute $S(\tilde e_k)$. We expand $\tilde e_k$ in the orthonormal basis $g_\ell /\Vert g_\ell \Vert $.
$$\tilde e_k=\sum_{\ell=1}^N(\tilde e_k\vert  g_\ell) \frac{ g_\ell}{\Vert g_\ell \Vert ^2}\ .$$ 
Moreover,
$$\Vert g_\ell \Vert ^2=\sum_{j=1}^N\frac{\lambda_j^2\nu_j^2}{(\lambda_j^2-\mu^2_\ell)^2}:=b_\ell \ .$$
Hence $$\tilde e_k=\sum_{\ell=1}^N \frac{ \lambda_k\nu_k \, {\rm e}^{i\varphi_k}}{ b_\ell(\lambda_k^2-\mu^2_\ell)}g_\ell$$
and, using Lemma \ref{spectralS}, 
$$S(\tilde e_k)=\sum_{\ell=1}^N \frac{ \lambda_k\nu_k \, {\rm e}^{i\varphi_k}}{b_\ell(\lambda_k^2-\mu^2_\ell)}{\mu_\ell}\, {\rm e}^{i\theta_\ell} h_\ell.$$
As $h_\ell=(H_u^2-\mu^2_\ell I)^{-1}(P_u(1))$ we get
$$(S(\tilde e_k)\vert \tilde e_j)=\sum_{\ell=1}^N \frac{ \lambda_k\nu_k\nu_j {\rm e}^{i(\varphi_k+\theta_\ell)}}{b_\ell(\lambda_j^2-\mu^2_\ell)(\lambda_k^2-\mu^2_\ell)}{\mu_\ell}.$$
Eventually, we obtain that 
$$u(z)=
X(I-zA)^{-1}Y$$
where 

\begin{eqnarray*}
X&:=&\left(\lambda_j\nu_j\, {\rm e}^{-i\varphi_j}\right)_{1\le j\le N}\\
Y&:=&\left(\nu_k \right)_{1\le k\le N}^T\end{eqnarray*}
and $A:=(A_{j,k})_{1\le j,k\le N}$ with $A_{j,k}=( \tilde e_j\vert S(\tilde e_k))$.

It remains to compute $\nu _j$ and $b_\ell $ in terms of $\lambda _k,\mu _m$. To this aim, we  shall use the generating function $J(x)$ defined by (\ref{jdex}), which in this case is given by
\begin{equation}\label{Jx2}
J(x)=1+x\sum _{j=1}^N\frac{\lambda _j^2\nu _j^2}{1-\lambda _j^2x}=\prod _{j=1}^N\frac{1-\mu^2 _jx}{1-\lambda _j^2x}\ .
\end{equation}
The second identity in (\ref{Jx2}) is  (\ref{ProdHmu}).The first one  comes from the expansion of $P_u(1)$ along the orthonormal basis $(e_1,\dots,e_N)$ :
\begin{eqnarray*}
J(x)&=&((I-xH_u^2)^{-1}(1)\vert 1)\\ &=&\Vert 1-P_u(1)\Vert ^2+(((I-xH_u^2)^{-1}(P_u(1))\vert P_u(1))\\&=&1-\sum _{j=1}^N\nu _j^2+\sum _{j=1}^N\frac{\nu _j^2}{1-\lambda _j^2x}\ .
\end{eqnarray*}
 Notice that 
 these identities  are valid for all complex values of $x$, except 
the poles $\lambda _j^{-2}$, $j=1,\dots ,N$. The value of $\nu _j^2$ is then obtained by computing the residue of $J(x)$ at the pole $1/\lambda _j^2$, while the value of $b_\ell $ is
given by
$$b_\ell =\frac1{\mu _\ell ^4}J'\left (\frac 1{\mu _\ell ^2}\right )\ .$$
\end{proof}

We shall now prove that $\chi _N$ is a diffeomorphism from $\mathcal M(N)_{\rm gen}$ onto $\Omega\times \T^{2N}$. The first step  is  to prove that $\chi _N$ is a local diffeomorphism.
 This will be a consequence of a first set of identities on the Poisson brackets  of the actions and the angles.

\subsection{First commutation identities}
First we recall some standard definitions. Given a smooth real-valued function $F$ on a finite dimensional  symplectic manifold $(\mathcal M,\omega )$, the Hamiltonian vector field of $F$
is the vector field $X_F$ on $\mathcal M$ defined by
$$\forall m\in \mathcal M, \forall h\in T_m\mathcal M, dF(m).h=\omega (h, X_F(m))\ .$$
Given two smooth real valued functions $F,G$, the Poisson bracket of $F$ and $G$ is
$$\{ F,G\} =dG.X_F=\omega (X_F,X_G)\ .$$
The above identity is generalized to complex valued functions $F, G$ by $\C $-bilinearity.
\begin{proposition}\label{InvolActionAngle}
For any $j,k\in\{1,\dots,N\}$ , one has
\begin{eqnarray*}
\{2\lambda_j^2,\varphi_k\}=\delta_{jk} &,&\;  \{2\mu^2_j,\varphi_k\}=0\ ,\\
\{2\lambda_j^2,\theta_k\}=0&,&\; \{2\mu^2_j,\theta_k\}=\delta_{jk}\ .
\end{eqnarray*}
\end{proposition}
In order to compute for instance $\{2\mu^2_j,\theta_k\}$ one has to differentiate $\theta_k$ along the direction of $X_{\mu^2_j}$. As  the expression of $X_{\mu^2 _j}$ is fairly complicated, we  use the "Szeg\" o hierarchy" , formed by the sequence of functions $J_{2n}$, which we studied in \cite{GG}. More precisely, we use the generating function $J(x)$ given by (\ref{Jx2}).  In the sequel, we shall restrict ourselves to real values of $x$, so that $J(x)$ is a real valued function.\s
 We proved in \cite{GG} that the Hamiltonian flow associated to $J(x)$ as a function of $u$ has a Lax pair, which we recall in the next statement. We set
 $$w(x):=(I-xH_u^2)^{-1}(1)\ .$$
\begin{theorem} [Szeg\"o hierarchy \cite{GG}, Theorem 8.1 and Corollary 8]\label{Jdex}
Let $s>\frac 12$. The map $u\mapsto J(x)$ is smooth on $H^s_+$
and its Hamiltonian vector field is given by
\begin{equation}\label{XJx}
X_{J(x)}(u)=\frac{x}{2i}w(x)H_uw(x)\ .
\end{equation}
 Moreover,
  the equation 
  \begin{equation}
\partial _tu=X_{J(x)}(u)\ 
\end{equation}
is equivalent to \begin{equation}\label{laxpairJx}
\partial _tH_u=[B_{u}^x,H_u]\ ,
\end{equation}
with
$$B_{u}^x(h)=\frac {x}{4i}\left (w(x)\Pi (\overline{w(x)} h)+xH_uw(x)\Pi (\overline{H_u(w)(x)} h)-x(h\vert H_uw(x))H_uw(x)\right )\ .$$
\end{theorem}
\begin{remark}\label{actioncom}
Notice that, since $B_u^x$ is skew-adjoint if $x$ is real, we infer that the spectrum of $H_u$ is conserved by the Hamiltonian flow of $J(x)$. Moreover, since 
\begin{equation}\label{Bx1}
B_u^x(1)=\frac{xJ(x)}{4i}w(x)\ ,
\end{equation}
we also deduce that the spectral measure of $H_u^2$ associated to vector $1$ is invariant.  Since,  by (\ref{eqmu}), the $\mu _m^2$ are the solutions in $\sigma $ of the 
equation $$((H_u^2-\sigma I)^{-1}H_u^2(1)\vert 1)=1,$$  we conclude that  the $\mu^2 _m$'s are also invariant.  We infer that the Poisson brackets of $J(x)$ with $\lambda _j^2$ or $\mu _m^2$ are
zero, which implies, in view of the expression (\ref{ProdHmu}), that the brackets of $\lambda _k^2$ or $\mu _\ell ^2$ with $\lambda _j^2$ or $\mu _m^2$ are zero.
\end{remark}
Thanks to this theorem, we can compute the Poisson brackets of $J(x)$ with the angles $\varphi_j$. The result is stated in the following lemma.
\begin{Lemma}\label{BracketsJxvarphi}
 $$\{J(x),\varphi_j\}=\frac 12 \frac{xJ(x)}{1-\lambda _j^2x}\ \ .$$
\end{Lemma}
\begin{proof}
Let us make $e_j$ evolve according to the Hamiltonian flow of $J(x)$. Taking the derivative of
 $H_u(e_j)=\lambda_j e_j$, we  get
\begin{eqnarray*}
\lambda_j\dot e_j&=&[B_{u}^x,H_u](e_j)+H_u(\dot e_j)\\
&=&\lambda_jB_{u}^x(e_j)-H_u(B_{u}^xe_j)+H_u(\dot e_j)
\end{eqnarray*}
Hence, $(H_u-\lambda_j I)(\dot e_j-B_u^xe_j)=0$, as by assumption $\ker(H_u-\lambda_jI)=\R e_j$, there exists $C_j\in \R$ so that
$$\dot e_j=B_{u}^x e_j+C_j e_j.$$
Using that ${\rm Re}(\dot e_j\vert e_j)=0$ as $e_j$ is normalized, and observing that $i(B_{u}^x e_j\vert e_j)$ is real-valued because of the skew-symmetry of $B_{u}^x$, we obtain $C_j=0$  .
Eventually, we have
$$\dot e_j=B_{u}^x e_j$$ and 
$$-4i(1\vert \dot e_j)=(1\vert 4iB_{u}^x(e_j))=(4iB_{u}^x(1)\vert e_j)=xJ(x)(w(x)\vert e_j)=\frac{xJ(x)}{1-\lambda _j^2x}(1\vert e_j)\ .$$ 
As a consequence 
$$\dot \varphi _j=\frac{d}{dt} \arg (1\vert e_j)^2=\frac 12\, \frac{xJ(x)}{1-\lambda _j^2x}\ .$$
\end{proof}
To compute the bracket with $\theta_m$, we are going to use the same method but we have to replace the Hankel operator $H_u$ by the shifted Hankel operator $K_u$. We first establish that there is also a Lax pair associated to $K_u$. We obtain it as a corollary of Theorem \ref{Jdex}.
\begin{corollary}\label{JdexKu}
 The equation 
  \begin{equation}
\partial _tu=X_{J(x)}(u)\ .
\end{equation}
 implies \begin{equation}\label{laxpairJxKu}
\partial _tK_u=[C_{u}^x,K_u]\ ,
\end{equation}
with
$$C_{u}^x(h)=\frac {x}{4i}\left (w(x)\Pi (\overline{w(x)} h)+xH_uw(x)\Pi (\overline{H_u(w)(x)} h)\right )\ .$$
\end{corollary}
\begin{proof}
One computes, by using Theorem \ref{Jdex},
\begin{eqnarray*}\partial _tK_u&=&\partial_t(H_uT_z)=[B_u^x,H_u]T_z=B_u^xK_u-H_uB_u^xT_z\\
&=& B_u^xK_u-K_uB_u^x+H_u[z,B_u^x].
\end{eqnarray*}
By the formula of $B_{u}^x$ given in Theorem \ref{Jdex}, and by the elementary identity $$\forall g\in L^2,\  \Pi (zg)-z\Pi (g)=(zg\vert 1), $$ 
we have
$$
[z,B_{u}^x](h)=-\, \frac x{4i}\left((zh\vert w)w+x(h\vert H_uw)zH_uw\right )\ ,
$$
so that 
\begin{eqnarray*}H_u[z,B_u^x](h)&=&\frac x{4i}\left((w-1\vert zh)H_uw+x(H_uw\vert h)K_uH_uw\right )\\
&=& \frac x{4i}\left((xH_u^2w\vert zh)H_uw+x(H_uw\vert h)K_uH_uw\right )\\
&=&\frac x{4i}\left(x(K_uH_uw\vert h)H_uw+x(H_uw\vert h)K_uH_uw\right )\\
&=& \frac {x^2}{4i}\left((K_uh\vert H_uw)H_uw+(H_uw\vert h)K_uH_uw\right )\\
&=&[D_u^x,K_u](h)
\end{eqnarray*}
where 
\begin{equation}\label{Dxu}
D_u^x(h)=\frac{ x^2}{4i}(h\vert H_uw)H_uw\ .
\end{equation}
Here we have used that, by definition of $w$, $w-1=xH_u^2w$. Coming back to the above expression for the derivative of $K_u$,
we obtain the claimed formula with $C_u^x=B_u^x+D_u^x$. This completes the proof.
\end{proof}
This Lax pair allow us to obtain the analogous of Lemma \ref{BracketsJxvarphi}.
\begin{Lemma}\label{BracketsJxtheta}
 $$\{J(x),\theta_m\}=-\, \frac 12\frac{xJ(x)}{1-\mu^2 _mx}\  .$$
\end{Lemma}
\begin{proof}
Let us look at  the evolution of $ f_m$ under the flow of $X_{J(x)}$.
Let us take the derivative of the equation $\mu_mf_m=K_u(f_m)$. We get, using the same arguments as for $\dot e_j$, that $\dot f_m=C_u^xf_m$. 
 We obtain
\begin{eqnarray*}
\frac d{dt}(u\vert f_m)&=&
(\dot u\vert f_m) +(u\vert \dot f_m)\\
&=& ([B_u^x,H_u](1)\vert f_m)+(u\vert C_u^xf_m)\\
&=&( B_u^x(u)\vert f_m)-(H_u\left(\frac{xJ(x)}{4i}w\right)\vert f_m)-(C_u^xu\vert f_m)\\
&=&-(D_u^x(u)\vert f_m)+\frac{xJ(x)}{4i} (H_uw\vert f_m)
\end{eqnarray*}
Using the above Formula \ref{Dxu} for $D_u^x$, we get
\begin{eqnarray*}
\frac d{dt}(u\vert f_m)&=&-\frac{x^2}{4i}(H_u^2w\vert 1)(H_uw\vert f_m)+\frac{xJ(x)}{4i} (H_uw\vert f_m)\\
&=&-\frac x{4i} (w-1\vert 1)(H_uw\vert f_m)+\frac{xJ(x)}{4i} (H_uw\vert f_m)\\
&=&\frac x{4i}(H_uw\vert f_m)=\frac x{4i}(u\vert f_m)(H_uw\vert g_m)\ .
\end{eqnarray*}

At this stage we observe that
$$(H_u^2-\mu^2 _mI)^{-1}H_uw=(H_u^2-\mu^2 _mI)^{-1}(I-xH_u^2)^{-1}u =\frac 1{1-\mu^2 _mx}(g_m+xH_uw)\ .$$
This yields
\begin{eqnarray*}
(H_uw\vert g_m)&=&((H_u^2-\mu^2_mI)^{-1} H_u(w)\vert u)\\
&=&\frac 1{1-\mu^2 _mx}(1+x(H_uw\vert u))=\frac {J(x)}{1-\mu^2 _mx}\  ,
\end{eqnarray*}
it implies 
$$\frac d{dt}(u\vert f_m)=(u\vert f_m)\frac x{4i}\frac{J(x)}{1-\mu_m^2 x}$$
and eventually 
$$\frac d {dt}{\rm arg}(u\vert f_m)^2=\frac d{dt}\theta_m =- \frac{x}{2}\frac{J(x)}{1-\mu_m^2x} \ .$$

\end{proof}

From  Lemma \ref{BracketsJxvarphi} and Lemma \ref{BracketsJxtheta} above, we easily deduce Proposition \ref{InvolActionAngle}. Indeed, from formula (\ref{Jx2}), we have
\begin{eqnarray*}
\{ J(x),\varphi _j\}&=&J(x)\sum _{k=1}^N  \left(\frac{x\{ \lambda _k^2,\varphi _j\} }{1-\lambda _k^2x}- \frac{x\{ \mu^2 _k,\varphi _j\} }{1-\mu^2 _kx}\right)\ ,\\
\{ J(x),\theta_m\}&=&J(x)\sum _{k=1}^N  \left(\frac{x\{ \lambda _k^2,\theta_m\} }{1-\lambda _k^2x}- \frac{x\{ \mu^2 _k,\theta_m\} }{1-\mu^2 _kx}\right)\ ,
\end{eqnarray*}
and the result follows from the comparison with the results of Lemma \ref{BracketsJxvarphi} and Lemma  \ref{BracketsJxtheta}.

 \begin{corollary}\label{localdiffeo}
 The mapping $\chi_N$ is a local diffeomorphism.
 \end{corollary}
 
 \begin{proof}
 Let us prove that the tangent map of $\chi$ is  invertible. Assume that  there exist $(\alpha_j)_{1\le j\le 2N}$ and $(\beta_j)_{1\le j\le 2N}$ so that
 $$\sum_{j=1}^N\alpha_jdI_j+\sum_{j=1}^N\alpha_{N+j}dL_j+\sum_{j=1}^N\beta_j d\varphi_j+\sum_{j=1}^N\beta_{N+j}d\theta_j=0.$$
 Since, by Remark \ref{actioncom} and Proposition \ref{InvolActionAngle}, $\{I_j,I_k\}=0$, $\{L_j,I_k\}=0$, $\{\theta_j,I_k\}=0$ and $\{\varphi_j,I_k\}=-\delta_{jk}$, by applying the above identity to $X_{I_k}$, we get $\beta_k=0$ for $k=1,\dots, N$. Doing the same with $X_{L_k}$, we get $\beta_k=0$ for $k=N+1,\dots 2N$. Applying this identity to $X_{\varphi_k}$ and then to $X_{\theta_k}$, we get $\alpha_k=0$, $k=1,\dots 2N$, this completes the proof.
 \end{proof}
 
\subsection{The surjectivity of the mapping $\chi _N$} In view of the inverse formula of Proposition \ref{InverseSpectral}, the surjectivity of $\chi _N$ is equivalent to the fact that, for every $(I,L,\varphi ,\theta )$ in $\Omega _N\times \T ^{2N}$, if $u$ is the right hand side of (\ref{ExplicitFormula}), then $\chi _N(u)=( I ,L ,\varphi ,\theta )$. Though the formulae are explicit, this fact is far from trivial and will lead to heavy calculations. Therefore we shall use another approach. Indeed we already know from Corollary \ref{localdiffeo} that $\chi _N$ is an open mapping. Since $\Omega _N\times \T^{2N}$ is connected, it suffices to prove that $\chi _N$ is proper hence closed to obtain that it is onto. 
Let us take a  sequence $(I^{(p)}, L^{(p)}, \varphi ^{(p)}, \theta ^{(p)})$ in $\Omega _N\times\T^{2N}$ which converges to $(I,L,\varphi, \theta )\in \Omega _N\times \T ^{2N}$, and such that, for every $p$, there exists  $u_p\in {\mathcal M}(N)_{{\rm gen}}$ such that 
$$\chi _N(u_p)=(I^{(p)}, L^{(p)}, \varphi ^{(p)}, \theta ^{(p)})\ .$$
Since 
$$\Vert u_p\Vert _{H^{1/2}}^2=\sum _{j=1}^N (\lambda _j^{(p)})^2=\frac 14 \sum _{j=1}^N(I^{(p)})^2$$
$(u_p)$  is a bounded sequence in $H^{1/2}_+$.  Up to extracting a subsequence, we may assume that $(u_p)_{p\in\N}$  converges weakly to some $u$ in $H^{1/2}_+$. At this stage we can appeal to Proposition 
\ref{compactness} and conclude that the convergence of $u_p$ to $u$ is strong and that
$$2\lambda _j^2(u)=I_j,\  j=1,\dots ,N,\ \ 2\mu _m^2(u)=L_m,\  m=1,\dots ,N, $$
with $ \lambda _j(u)=0$ if $j>N$, $\mu _m(u)=0$ if   $m>N$.
Therefore  $u\in \mathcal M(N)_{{\rm gen}}\ .$ This completes the proof of the surjectivity of $\chi _N$.

\subsection { The remaining commutation identities}
At this stage we proved that $\chi _N$ is a global diffeomorphism.
We are going to show that it is symplectic. In view of Proposition \ref{InvolActionAngle}, it suffices to prove that the Poisson brackets of $\{\varphi_j,\varphi_k\},$ $\{\theta_k,\theta_{k'}\}$ and $\{\varphi_j,\theta_\ell\}$ 
cancel. We first remark that, thanks to the first commutations properties and to the Jacobi identity, these brackets  are functions of the actions $(I,L)$ only. Indeed, applying 
$$\{f,\{\varphi_j,\varphi_k\}\}+\{\varphi_j,\{\varphi_k,f\}\}+\{\varphi_k,\{f,\varphi_j\}\}=0$$ 
to  $f=I_\ell $ and $f=L_m$, we obtain, in view of Proposition \ref{InvolActionAngle},
 $$\{ I_\ell ,\{\varphi_j,\varphi_k\}\}=\{ L_m ,\{\varphi_j,\varphi_k\}\}=0.$$
Writing $\{\varphi_j,\varphi_k\}=g(I,L,\varphi, \theta )$, we infer, by Remark \ref{actioncom} and Proposition \ref{InvolActionAngle},
$$\frac{\partial g}{\partial \varphi _\ell }=\frac{\partial g}{\partial \theta _m}=0\ .$$
The same holds for $\{\theta_k,\theta_{k'}\}$ and $\{\varphi_j,\theta_\ell\}$ .

We now prove the remaining commutation laws by first establishing the following result. Recall that $J_n(u):=(H_u^n(1)\vert 1)$.
\begin{Lemma}\label{J3J1}
One has
$\{J_3,J_1\}=-\frac i2 J_1^2$.
\end{Lemma}
\begin{proof}
From the definition of $J_1$, one has
$J_1(u)=(u|1)$ so that $dJ_1(u)(h)=(h|1)$. On the other hand, $J_3(u)=(H_u^2(u)|1)$ so that $$dJ_3(u)(h)=(H_hH_u^2(1)+H_uH_hH_u(1)+H_u^2H_h(1)|1)=2(h|H_u^2(1))+(u^2|h).$$
As $dJ_3(u)(h)=4\text{Im}(h|X_{\text{Re}J_3})+4i\text{Im}(h\vert X_{\text{Im}J_3})$, it implies that
\begin{eqnarray*}
X_{\text{Re}J_3}=-\frac i2 H_u^2(1)-\frac i4u^2\\
X_{\text{Im}J_3}=\frac 12 H_u^2(1)-\frac 14u^2.
\end{eqnarray*}
Thus, one obtains

$\{J_3,J_1\}=dJ_1(X_{\text{Re}J_3})+idJ_1(X_{\text{Im}J_3})=-\frac i2J_1^2$ and the lemma is proved.

\end{proof}
As a corollary, we get the following commutation laws.
\begin{corollary}
For any $j,k$, $\{\varphi_j,\varphi_k\}=0$.
\end{corollary}
\begin{proof}
From the definitions of $J_1$ and $J_3$, we have
$J_1=\sum_j \lambda_j\nu_j^2e^{-i\varphi_j}$ and $J_3=\sum_k\lambda_k^3\nu_k^2e^{-i\varphi_k}$ so that
$$
\{J_3,J_1\}=\sum_{j,k} e^{-i(\varphi_j+\varphi_k)}[-i\{\lambda_k^3\nu_k^2,\varphi_j\}\lambda _j\nu _j^2+i\{\lambda_j\nu_j^2,\varphi_k\}\lambda _k^3\nu _k^2-\{\varphi_j,\varphi_k\}\lambda_j\nu_j^2\lambda_k^3\nu_k^2]\ .
$$
On the other hand, by  Lemma \ref{J3J1}, one also has
$$
\{J_3,J_1\}=-\frac i2J_1^2=-\frac i2\sum_{j,k}\lambda_j\lambda_k\nu_j^2\nu_k^2{\rm e}^{-i(\varphi_j+\varphi_k)}\ .$$
As the commutators $\{\varphi_j,\varphi_k\}$, $\{\lambda_k\nu_k^2,\varphi_j\}$ and $\{\lambda_k^3\nu_k^3,\varphi_j\}$ only depend
 on the actions $(I,L)$, we can identify the Fourier coefficients of the  function 
$\{ J_3,J_1\} $ as a trigonometric polynomial in the angle variables. We focus on the Fourier coefficient for $j\ne k$.  Since $\{ \lambda _k,\varphi _j\} =0$, one gets
\begin{equation}\label{identif}
(\lambda _k^2-\lambda _j^2)\left [-\{\varphi_j,\varphi_k\} +i\left(\frac{\{ \nu_j^2,\varphi_k\}}{\nu _j^2}-\frac{\{\nu_k^2,\varphi_j\}}{\nu _k^2}\right)\right ]=-i \ .
\end{equation}
Taking the real part of both sides, we conclude 
$$\{\varphi_j,\varphi_k\}=0\ .$$

\end{proof}

We now compute the commutation laws between the $\varphi_j$'s and the $\theta_k$'s. We shall make use of the functionals $N_{2n+1}(u)=(zu|H_u^{2n}(1))$. Recall that the operator $K_u^2$ has the $\mu^2_k$'s as eigenvalues with associated eigenfunctions 
$g_k=(H_u^2-\mu^2_k I)^{-1}u$, with $\Vert g_k\Vert ^2=b_k=\sum_j\frac{\lambda_j^2\nu_j^2}{(\lambda_j^2-\mu^2_k)^2}$. Hence, by Formula (\ref{SFm}), $$P_u(zu)=\sum_k\frac 1{b_k}S(g_k)=\sum_k\frac{ \mu_k \, {\rm e}^{i\theta_k}}{b_k} h_k$$ where 
$h_k=(H_u^2-\mu^2_k I)^{-1} P_u(1)$. Hence we have  
$$N_{2n+1}(u)=\sum_k\frac{\mu_k \, {\rm e}^{i\theta_k}}{b_k}(h_k|H_u^{2n}(1))=\sum_k\frac{\mu_k \, {\rm e}^{i\theta_k}}{b_k}P_n(\mu_k)$$ where $P_n(\mu)=\sum_{j}\frac{\lambda_j^{2n}\nu_j^2}{\lambda_j^2-\mu^2}$.
We first compute the commutator of $N_3$ with $J_1$.
\begin{Lemma}
One has
$\{N_3,J_1\}=0$.
\end{Lemma}
\begin{proof}
As $N_3(u)=(zu|H_u^2(1))$, one has 
\begin{eqnarray*}
dN_3(u)(h)&=&(zh|H_u^2(1))+(zu|H_u(h)+H_h(u))\\
&=&2(h| H_u(zu))+(zu^2|h).
\end{eqnarray*}
So, one gets
\begin{eqnarray*}
X_{\text{Re}N_3}=-\frac i2 H_u(zu)-\frac i4zu^2\ ,\\
X_{\text{Im} N_3}=\frac 12 H_u(zu)-\frac 14zu^2\ .
\end{eqnarray*} 
It implies that 
$$\{N_3,J_1\}=dJ_1.X_{\text{Re}N_3}+idJ_1.X_{\text{Im}N_3}=-\frac i2(zu^2|1)=0.$$
\end{proof}
As a corollary, one gets
\begin{corollary}
For any $k$ and $j$, $\{\theta_k,\varphi_j\}=0$.
\end{corollary}
\begin{proof}
The proof follows the same lines as before. One writes that $\{N_3,J_1\}=0$. One has
\begin{eqnarray*}
0&=&\{N_3,J_1\}=\sum_{j,k}\{\frac 1{b_k}\mu_k \, {\rm e}^{i\theta_k},\lambda_j\nu_j^2e^{-i\varphi_j}\}\\
&=&\sum_{j,k}\, {\rm e}^{i(\theta_k-\varphi_j)}[-i(\{\frac{\mu_k}{b_k},\varphi_j\}+\{\lambda_j\nu_j^2,\theta_k\})+\frac{\mu_k}{b_k
}\lambda_j\nu_j^2\{\theta_k,\varphi_j\}].
\end{eqnarray*}
By cancelling the real  part of the Fourier coefficient, one gets the result.

\end{proof}
By computing the commutator of $N_3$ and $N_5$, one gets as well that $\{\theta_k,\theta_{k'}\}=0$.
Let us give the proof for completeness.
As $N_5(u)=(zu\vert H_u^4(1))$, we have 
\begin{eqnarray*}
dN_5(u)(h)&=&(zh\vert H^4_u(1))+(zu\vert H_hH^3_u(1)+H_uH_hH_u^2(1)+\\&+&H_u^2H_hH_u(1)+H_u^3(h))\\
&=&(h\vert H_u(zH_u^2(u))+H_u(u)H_u(zu)+H_u^3(zu))+\\&+&(zuH_u^3(1)+H_u^2(zu)u\vert h).
\end{eqnarray*}
So, using the expression of $X_{{\rm Re}N_3}$ and of $X_{{\rm Im}N_3}$, we get
\begin{eqnarray*}
\{N_3,N_5\}&=&dN_5(u)(X_{{\rm Re}N_3})+idN_5(u)(X_{{\rm Im}N_3})\\
&=&-\frac {i}2(zu^2\vert H_u(zH_u^2(u))+H_u(u)H_u(zu)+H_u^3(zu))\\
&+&i(zuH_u^2(u)+uH_u^2(zu)\vert H_u(zu))\\
&=&-\frac i2\left[ (zH_u^2(u)+H_u^2(zu)\vert H_u(zu^2))+(zu^2\vert H_u(u)H_u(zu))\right]\\
&+&i(zuH_u^2(u)+uH_u^2(zu)\vert H_u(zu)).
\end{eqnarray*}
Applying the formulae
$$(zf\vert g)=(z\Pi (f)\vert \Pi (g))+(\Pi (\overline g)\vert \Pi (\overline z \overline f))\ ,\ H_{H_u(a)}(b)=H_u(ab)\ ,$$
 we have
\begin{eqnarray*}
(zu^2\vert H_u(u)H_u(zu))&=&(zu\overline{H_u(u)}\vert \overline uH_u(zu))\\
&=&(zH_u^2(u)\vert H_u(zu^2))+(H_u^2(zu)\vert H_u(zu^2)), \\
(uH_u^2(zu)\vert H_u(zu))&=&(H_u^2(zu)\vert \overline u H_u(zu))= (H_u^2(zu)\vert H_u(zu^2)), \\
(zuH_u^2(u)\vert H_u(zu))&=&(zH_u^2(u)\vert H_u(zu^2))\ ,
\end{eqnarray*}
so that eventually
$$\{N_3,N_5\}=0.$$
On the other hand, we have, as $$N_3(u)=\sum_{\ell}\frac{\mu_\ell}{b_\ell}\, {\rm e}^{i\theta_\ell}\text{ and }N_5(u)=\sum_k\frac{\mu^2_k+J_2}{b_k}\mu_k\, {\rm e}^{i\theta_k},$$
\begin{eqnarray*}
0=\{N_3,N_5\}&=&\sum_{\ell,k}\, {\rm e}^{i(\theta_\ell+\theta_k)}\left[i\{\frac{ \mu_\ell}{b_\ell},\theta_k\}\frac{\mu_k^2+J_2}{b_k}\mu_k\right.\\
&-&\left. i\{\frac{ \mu_k}{b_k}(\mu^2_k+J_2),\theta_\ell\}\frac{\mu_\ell}{b_\ell}
-\frac{{\mu_\ell\mu_k}}{b_\ell b_k}(\mu^2_k+J_2)\{\theta_\ell,\theta_k\}\right]
\end{eqnarray*}

Now, as before, one can cancel the real part of the Fourier coefficients  to obtain $$\{\theta_\ell,\theta_k\}\frac{{\mu_\ell\mu_k}}{b_\ell b_k}(\mu_k^2-\mu_\ell^2)=0
$$ and hence, $\{\theta_\ell,\theta_k\}=0$ . 

We have therefore proved all the commutation relations between our action angle variables.This proves that $\chi _N$ is a symplectomorphism and
completes the proof of Theorem \ref{TheoDiffeo}.

\subsection{The explicit solution of the cubic Szeg\"o equation}

We first prove Corollary \ref{szegosolution}. 

\begin{proof} 
Let us compute $$\Delta_4={\rm Tr}(H_u^4)-{\rm Tr}(K_u^4)={\rm Tr}(H_u^4)-{\rm Tr}((H_u^2-(\cdot\vert u)u)^2)$$ in terms of $J_2$ and $J_4$. We get $\Delta_4=2J_4-J_2^2$. On the other hand, we already pointed out that $2J_4-J_2^2=\Vert u\Vert_{L^4}^4$. Since the cubic Szeg\"o equation on $\mathcal M(N)$ is  the Hamiltonian system associated to the functional $E(u)=\Vert u\Vert_{L^4}^4$and to the symplectic form $ \omega$, and since  $\chi_N$ is a symplectomorphism, we obtain that the cubic Szeg\"o equation is equivalent to the Hamiltonian system associated to 
$$E(I,L,\varphi,\theta)=\frac 14\sum_{j=1}^N(I_j^2-L_j^2).$$
As the new coordinates are symplectic, we obtain that the cubic Szeg\"o equation is equivalent to the system
$$\left\{\begin{array}{cc}
\dot I_j=0\ , \; \dot L_m=0\\
\dot\varphi_j=\frac12I_j\ ,\ \dot \theta_m=-\frac 12 L_m\end{array}\right.$$
\end{proof}
\begin{remark} Notice that the above system is explicitely solvable, and therefore that we reduced the cubic Szeg\"o equation to a spectral analysis of the Hankel operator associated to the Cauchy datum $u_0$.
In \cite{GG}, section 4.1, we observed that the cubic Szeg\"o equation  on $\mathcal M(N)$ could be written as a system of $2N$ ordinary differential equations in the variables given by the poles and the residues of the rational function $u$. Therefore the above corollary provides an explicit  resolution of this system.
\end{remark}
\end{section}

\begin{section}{Extension to the infinite dimension}\label{infini}
In this section, we prove Theorem \ref{TheoHomeo}.
We begin with proving the genericity of the set $H^{1/2}_{+,{\rm gen}}$.
\begin{Lemma}
The set  $H^{1/2}_{+,\rm{gen}}$ is a dense $G_\delta$ subset of $H^{1/2}_+$.
\end{Lemma}
\begin{proof}
Let us consider the set $\mathcal U_N$ which consists of functions $u\in H^{1/2}_+$ such that the first $N$ eigenvalues  of $H_u^2$ are simple,
and such that, for any $j\in\{1,\dots N\}$, $\nu_j:=\Vert P_j(1)\Vert\neq 0$.
This set is obviously open in $H^{1/2}_+$. It is also dense in $H^{1/2}_{+,\rm{gen}}$ since any element $u$ in $H^{1/2}_{+,\rm{gen}}$ may be approximated by an element in $\mathcal M(N')$, $N'>N$, which can be itself approximated by an element in $\mathcal M(N')_{\rm gen}\subset \mathcal U_N$, since $N'\ge N$. Eventually, $H^{1/2}_{+,\rm{gen}}$ is the intersection of the $\mathcal U_N$'s which are open and dense, hence $H^{1/2}_{+,\rm{gen}}$ is a dense $G_\delta$ set.
\end{proof}
We can now begin the proof of Theorem \ref{TheoHomeo}. First of all, it is clear that, because of the simplicity assumption on the eigenvalues $\lambda _j^2$ and $\mu _m^2$,
 each function $\zeta _j$ and $\gamma _m$ is continuous.
Let $(u_n)$ in $ H^{\frac 12}_{+,{\rm gen}}$ be a sequence so that $u_n$ converges to some $u$ in the topology of $H^{\frac 12}.$   Since $H_{u_n}$ converges to $H_u$ in the Hilbert-Schmidt norm,
 the $\ell ^2$ norm of  $(\lambda_j(u_n))$ tends to the $\ell ^2$ norm of  $(\lambda_j(u))$ in $\ell^2$. As $K_{u_n}$ tends $K_u$ in the Hilbert-Schmidt norm as well, the $\ell ^2$ norm of $(\gamma_j(u_n))$ tends to 
 the $\ell ^2$ norm of $(\gamma_j(u))$. This implies that $\chi (u_n)$ tends  to $\chi (u)$ in $\ell ^2\times \ell ^2$.
\s  We now show that $\chi $  is a homeomorphism.
Let us first prove that $\chi$ is onto. Let $((\zeta_j),(\gamma_m))\in \Xi$. As $\chi _N$ is onto on $\mathcal M(N)_{\rm gen}$, for any $N\ge 1$, to $(\vert \zeta_j\vert ,\vert \gamma_j\vert , \varphi _j, \theta _j)_{1\le j\le N}$ corresponds a 
unique $u_N\in\mathcal M(N)_{\rm gen}.$  Since $$\Vert u_N\Vert_{H^{1/2}}^2={\rm Tr}(H_{u_N}^2)=\sum_{j=1}^N\lambda_j^2\to \sum_{j=1}^\infty \lambda_j^2\ ,$$ $(u_N)$ is bounded in $H^{1/2}$ and  there exists a subsequence, still denoted by $(u_N)$, which converges weakly to some $u\in H^{1/2}_+$. Appealing to Proposition \ref{compactness}, we infer that  $u$ is the strong limit of $(u_N)$ in $H^{1/2}$ and that $u$ belongs to $H^{1/2}_{+,{\rm gen}}$ with $\chi (u)=((\zeta_j),(\gamma_m))$.

Let us  prove that $\chi$ is one-to-one. Again we use the formula (\ref{InverseSpectral}),
 $$u(z)=(u\vert(I-\overline zS)^{-1}P_u(1))\ .$$ 
Arguing as in Section 3,  it is easy to check that  Formula (\ref{ExplicitFormula}) may be extended here so that $u$ is uniquely determined from the data in $\Xi$,
as shown by the following result.
\begin{proposition}\label{inversespectralinfini}
If $\chi (u)=((\zeta _j)_{j\ge 1}, (\gamma _m)_{m\ge 1})$, then
\begin{equation}u(z)=X.(I-zA)^{-1}Y\end{equation}
where
\begin{eqnarray*}
X&:=&\left(\nu_j\zeta _j\right)_{ j\ge 1}\\
Y&:=&\left(\nu_k \right)_{k\ge 1}^T\ ,\end{eqnarray*}
 $A:=(A_{j,k})_{ j,k\ge 1}$ is given by  
$$A_{j,k}=\sum_{\ell =1}^\infty  \frac{ \nu_j \nu_k \zeta _k\gamma _\ell }{b_\ell(\vert \zeta_j\vert ^2-\vert \gamma _\ell \vert ^2)(\vert \zeta _k\vert ^2-\vert \gamma _\ell \vert ^2)}\ ,$$
and 
\begin{eqnarray*}
\nu_j&=&\frac 1{\vert \zeta_j\vert }\frac{\prod_{k=1}^\infty(\vert \zeta_j\vert ^2-\vert \gamma _k\vert ^2)^{1/2}}{\prod_{k\neq j}(\vert \zeta_j\vert ^2-\vert \zeta_k\vert ^2)^{1/2}}\ ,\\
b_\ell&=&\sum_{j=1}^\infty \frac{\vert \zeta_j\vert ^2\nu_j^2}{(\vert \zeta_j\vert ^2-\vert \gamma _\ell\vert ^2)^2}=\frac 1{\vert \zeta _\ell \vert ^2-\vert \gamma  _\ell \vert ^2}
\prod _{k\ne \ell} \frac{\vert \gamma  _\ell \vert ^2-\vert \gamma  _k\vert ^2}{\vert \gamma  _\ell \vert ^2-\vert \zeta _k\vert ^2}\ .
\end{eqnarray*}
\end{proposition}
\begin{proof}
Since it is very similar to the proof of Proposition \ref{inversespectral}, we only indicate the new features. Denote by $\mathcal R_u$ the closure of the range of $H_u$.
The main difference relies on the spectral theory of $K_u$ on $\mathcal R_u$. Indeed, if $u\in H^{1/2}_{+,{\rm gen}}$, it may happen that $K_u$
has a kernel in $\mathcal R_u$, which is equivalent, as we noticed in Subsection \ref{SKu},  to the existence of $g_0\in \mathcal R_u$ such that $H_ug_0=1$. In this case,
an orthogonal basis of the Hilbert space $\mathcal R_u$ is given by the sequence $(g_m)_{m\ge 0}$, where $g_m, m\ge 1,$ is given by the formula (\ref{Fm}), and $g_0$ is as 
above. However it turns out that the existence of $g_0$ does not affect the formulae in Proposition \ref{inversespectralinfini}. Indeed, since $g_0\in \mathcal R_u$ and
$K_ug_0=0=H_uSg_0$, we infer $Sg_0=0$, hence, with the notation of Proposition \ref{inversespectral}, the expression of $S(\tilde e_k)$ is still
$$S(\tilde e_k)=\sum_{\ell=1}^\infty \frac{ \lambda_k\nu_k \, {\rm e}^{i\varphi_k}}{b_\ell(\lambda_k^2-\mu^2_\ell)}{\mu_\ell}\, {\rm e}^{i\theta_\ell} h_\ell \ ,\ b_\ell :=\Vert g_\ell \Vert ^2\ ,$$
and the expression of
$$A_{j,k}=(\tilde e_j\vert S(\tilde e_k))$$
then follows for every $j,k\ge 1$.
\end{proof}

It remains  to check that $\chi^{-1}$ is continuous on $\Xi$, that is if $\chi(u_p)$ tends to $\chi(u)$ then $u_p$ tends to $u$ in $H^{1/2}$.
First, as $\chi(u_p)$ converges, the sequence $(u_p)$ is bounded in $H^{1/2}$ and hence, admits a convergent subsequence which weakly converges to some $v$. 
Appealing again to Proposition \ref{compactness}, we conclude that $u_p$ converges strongly to $v$. 
As $\chi$ is continuous and one-to-one,  we have $u=v$. 

Finally, the evolution formulae of $\zeta _j$ and of $\gamma _m$ for the cubic Szeg\"o equation (\ref{szego}) are immediate consequences
 of similar formulae for $u\in \mathcal M(N)$ derived in Corollary \ref{szegosolution}, combined with the approximation of $u$ by elements $u_N$ in $\mathcal M(N)_{{\rm gen}}$, and the 
 continuity of the flow map of (\ref{szego}) on $H^{1/2}_+$, as proved in Theorem 2.1 of \cite{GG}.
This completes the proof.
\end{section}

\begin{section}{Inverse spectral problems for Hankel operators}\label{invhankel}
As a byproduct of the existence of  the diffeomorphism $\chi _N$ and of the homeomorphism $\chi $, we  first prove Corollary \ref{inverseHankel}.
\begin{proof} 
Denote by  $\mathcal F:u\in L^2_+\mapsto c=(\hat u(n))_{n\ge 0}\in \ell ^2(\N )$ the Fourier transform. Notice that $\mathcal F$ realizes an isomorphism from $H^{1/2}_+$ onto ${\bf h}^{1/2}$.
Moreover,  it easy to check that
 $$\mathcal F^{-1}\Gamma _c\Gamma _c^*\mathcal F=H_u^2\ ,\ \mathcal F^{-1}\tilde \Gamma _c\tilde \Gamma _c^*\mathcal F=K_u^2\ .$$
  Therefore, the set of sequences $c\in {\bf h}^{1/2}$ such that  $\Gamma _c$ has rank $N$ and admits   $\lambda_j\ ,\ 1\le j\le N,$ as simple singular values,
 and such that  $\tilde \Gamma _c$ has rank $N$ and admits  $\mu _j,\ 1\le j\le N,$ as simple singular values, is sent by $\mathcal F^{-1}$ onto 
$$\chi _N^{-1}((I_1,\dots ,I_N,L_1,\dots ,L_N)\times \T ^{2N})\ ,$$
with
$$I_j:=2\lambda _j^2\ ,\ L_m:=2\mu _m^2\ .$$
The same argument applies in the infinite dimensional case.This completes the proof.
\end{proof}
Restricting to the case of selfadjoint Hankel operators will give us the proof of Corollary \ref{Hankelauto} as follows.
\begin{proof}
Via the Fourier transformation $\mathcal F$, 
$$L^2_{+,r}=\{ h\in L^2_+\, :\,  \forall n\in \N , \hat h(n)\in \R \}\  .$$
 identifies to $\ell ^2_\R (\N)$, and the operators $H_u, K_u$ with $u\in H^{1/2}_+\cap L^2_{+,r}$ respectively identify to
$\Gamma _c, \tilde \Gamma _c$ with $c=\mathcal Fu $. Moreover, for every $(I,L)\in \Omega _N$, one easily checks that $${\bf T}(I,L)\cap L^2_{+,r}=\chi _N^{-1}((I,L)\times \{ 0,\pi \} ^{2N}\} \ ,$$ 
and, if $u$ belongs to this set,  the non zero eigenvalues of $H_u$ (resp. $K_u$) on $L^2_{+,r}$ are
$$\zeta _1=\lambda _1{\rm e}^{-i\varphi _1},\dots ,\zeta _n=\lambda _N{\rm e}^{-i\varphi _N} ({\rm resp.}\  \gamma _1=\mu _1{\rm e}^{-i\theta _1},\dots ,\gamma _N=\mu _N{\rm e}^{-i\theta _N})\ .$$
Indeed, on the one hand $\chi _N^{-1}((I,L)\times \{ 0,\pi \} ^{2N}\}\subset {\bf T}(I,L)\cap L^2_{+,r}$ by Proposition \ref{inversespectral}. On the other hand, if $u\in {\bf T}(I,L)\cap L^2_{+,r}$, the operator $H_u$ is selfadjoint on $L^2_{+,r}$, hence has real eigenvalues $\zeta _1,\dots ,\zeta _N$ with $\vert \zeta _j\vert=\lambda _j$. The corresponding normalized eigenvectors $\tilde e_j$ in $L^2_{+,r}$ satisfy
$$H_u(\tilde e _j)=\zeta _j\tilde e_j\ ,$$
therefore either $\tilde e_j=\pm e_j$ and $\varphi _j=\arg (1\vert \tilde e_j)^2=0$ if $\zeta _j=\lambda _j$, or $\tilde e_j=\pm ie_j$ and $\varphi _j=\arg [-(1\vert \tilde e_j)^2]=\pi $ if $\zeta _j=-\lambda _j$ . 
The same holds for $K_u$.
The same argument applies in the infinite dimensional case. This completes the proof.
\end{proof}
\begin{remark}\label{cn}
Notice that, in addition to Corollaries \ref{inverseHankel} and \ref{Hankelauto}, the solutions $c$ are given by
$$c_n=XA^nY\ ,$$
with the notation of Proposition \ref{inversespectral} in the finite rank case, and Proposition \ref{inversespectralinfini} in the infinite rank case.
\end{remark}
\end{section}

\begin{section}{Stability of Invariant Tori}\label{stabilitysection}

In this section, we prove Theorem \ref{stability}, which we state again for the convenience of the reader. 

\begin{theorem}
For $n=1,\dots ,2N$, define
\begin{equation}\label{j2nT}
j_{2n}=\sum _{j=1}^N2^{-n}I_j^n\left (1-\frac {L_j}{I_j}\right )\prod _{k\ne j}\left (\frac{L_k-I_j}{I_k-I_j}\right )\ .
\end{equation}
Then ${\bf T}(I_1,\dots,I_N,L_1,\dots,L_N)$ is the set of the solutions in $H^{1/2}_+$ of the minimization problem
$$\inf \{ M(u)\;  : \; J_{2n}(u)=j_{2n}\ ,\ n=1,\dots ,2N\} \  .$$
Consequently, ${\bf T}:={\bf T}(I_1,\dots,I_N,L_1,\dots,L_N)$ is stable under the evolution of (\ref{szego}), in the sense that,
for every $\varepsilon >0$, there exists $\delta >0$ such that, if 
$$\inf _{v\in{\bf T}}\Vert u_0-v\Vert _{H^{1/2}}\le \delta \ ,$$
then the solution $u$ of (\ref{szego}) with $u(0)=u_0$ satisfies 
$$\sup _{t\in \R }\inf _{v\in{\bf T}}\Vert u(t)-v\Vert _{H^{1/2}}\le \varepsilon \  .$$
\end{theorem}

\begin{proof}
First of all, notice that Formula (\ref{j2nT}) expresses the common value of $J_{2n}(u)$ as $u\in {\bf T}(I_1,\dots,I_N,L_1,\dots,L_N)$, in view of formulae (\ref{nulambdamu}) and  
$$J_{2n}=\sum _{j=1}^N\lambda _j^{2n}\nu _j^2$$
with $I_j=2\lambda _j^2, L_m=2\mu _m^2$.
\s
Let us assume that the Lagrangian torus $${\bf T}:={\bf T}(I_1,\dots,I_N,L_1,\dots,L_N)$$ is the set of solutions in $H^{1/2}_+$ of the minimization problem
\begin{equation}
\label{infM(u)}
\inf\{M(u)\;:\; J_{2n}(u)=j_{2n}, \; n=1,\dots,2N\}:=m
\end{equation} 
where the $j_{2n}$'s are given by formula (\ref{j2nT}) and let us prove that it implies the stability.

Let $u_0^{(n)}$ so that $\inf_{v\in {\bf T}}\Vert u_0^{(n)}-v\Vert_{H^{1/2}}$ tends to zero as $n$ goes to infinity. We are going to show that the solutions $u^{(n)}$ of the cubic Szeg\"o equation with $u^{(n)}(0)=u_0^{(n)}$ are such that $\sup_{t\in \R}\inf_{v\in {\bf T}}\Vert u^{(n)}(t)-v\Vert_{H^{1/2}}$ tends as well to zero as $n$ goes to infinity.
As the functionals $u\mapsto J_{2k}(u)$ are invariant under the cubic Szeg\"o flow and are continuous for the weak topology of $H^{1/2}$ we get that $J_{2k}(u^{(n)}(t))=J_{2k}(u^{(n)}_0)$ tends to $j_{2k}$.  Similarly, since $M(u)$ is a conservation law, $u^{(n)}$ is bounded in $H^{1/2}$ and $M(u^{(n)})$ tends to $m$.  Moreover, given any sequence $(t_n)$ of real numbers, the sequence $(u^{(n)}(t_n))$  has a subsequence which converges weakly  to some $u\in H^{1/2}_+$. By the weak continuity of the $J_{2k}$ and the weak semi-continuity of $M$, $J_{2k}(u)=j_{2k}$ and $M(u)\le m$. Hence, since $\bf T$ is the solution of the minimization problem,  $M(u)=m$,  $u_n(t_n)$ converges strongly to $u$ and $u$ belongs to ${\bf T}$. This gives the stability. 
\s
It remains to prove that the set of minimizers is ${\bf T}$. Recall that $H_u^{2k}(1)$, $k=1,\dots,N$ are linearly independent if and only if the Gram determinant
$$\det(J_{2(n+m)}(u))_{1\le n,m\le 2N}$$ is non-zero.
By the choice of the sequence $\{j_{2n}\}_{1\le n\le 2N}$, there exists $u\in \mathcal M(N)_{\rm gen}$ so that $J_{2n}(u)=j_{2n}$, $1\le n\le 2N$ --- any $u\in\chi _N^{-1}((I_1,\dots,I_N,L_1,\dots,L_N)\times \T ^{2N})$ is convenient. Hence, the determinant $$\det(j_{2(n+m)})_{1\le n,m\le 2N}$$ is different from zero. Since $H_u$ is one to one on its range, it follows that if $u$ satisfies $J_{2n}(u)=j_{2n}$, $1\le n\le 2N$ then $u,H^2_u(u),\dots,H_u^{2(N-1)}(u)$ are independent. As a first step, the following proposition implies that the set of functions $u$ with $J_{2n}(u)=j_{2n}$ with $M(u)$ minimal is a subset of $\mathcal M(N)$.

\begin{proposition}\label{M(u)det}
Let $u\in H^{1/2}_+$ and $N\ge1$ so that $u,H^2_u(u),\dots,H^{2(N-1)}_u(u)$ are independent.
Then the following inequality holds
$$M(u)\ge\frac{\det\left((J_{2(k+\ell+1)}(u))_{0\le k, \ell \le N-1}, (J_{2(k+N+1)}(u))_{0\le k\le N-1}\right)}{\det(J_{2(k+\ell+1)}(u))_{0\le k,\ell\le N-1})}$$
with equality if and only if $u\in{\mathcal M}(N)$.
\end{proposition}

\begin{proof}
This statement is a direct consequence of the following lemma with $A=H^2_u$ and $e=u$.
\end{proof}
\begin{Lemma}\label{traceA}
Let $A$ be a  trace class positive self-adjoint operator defined on a Hilbert space $\mathcal H$ and let  $e\in\mathcal H$, $N\ge 1$. Assume that $A(e),A^2(e),\dots,A^N(e)$ are independent.
Then,

$${\rm Tr}(A)\ge\frac{\det\left((A^{k+\ell}(e),e)_{0\le k\le N-1\atop 0\le \ell \le N-2}, (A^{k+N}(e),e)_{0\le k\le N-1}\right)}{\det\left((A^{k+\ell}(e),e)_{0\le k,\ell\le N-1})\right)}$$
with equality if and only if the range of $A$ is $N$ dimensional  and $e$ belongs to the range of $A$.
\end{Lemma}

\begin{proof}

Denote by $V$ the space spanned by $e,A(e),\dots, A^{N-1}e$. Let $P$ be the orthogonal projector from $\mathcal H$ to $V$. Let $\tilde A=PAP$ then $\tilde A$ is positive self adjoint and $\text{Tr}(A)\ge \text{Tr}(\tilde A)$. In fact, one has
$$\text{Tr}(\tilde A)=\text{Tr}(P^2A)=\text{Tr}(PA)$$
so that
$$\text{Tr}(A)-\text{Tr}(\tilde A)=\text{Tr}((I-P)A)=\text{Tr}((I-P)^2A)=\text{Tr}((I-P)A(I-P))\ge 0.$$

By definition, $\tilde A$ is at most of range $N$ so that by Cayley-Hamilton, there exist $\sigma_1=\text{Tr}(\tilde A),\sigma _2\dots,\sigma_N$ so that
$$(\tilde A)^N=\sum_{j=1}^N(-1)^{j-1}\sigma_j(\tilde A)^{N-j}.$$
In particular,
$$(\tilde A)^N(e)=\sum_{j=1}^N(-1)^{j-1}\sigma_j(\tilde A)^{N-j}(e)$$ so that
$$PA^N(e)=\sum_{j=1}^N(-1)^{j-1}\sigma_j A^{N-j}(e)$$ and taking the scalar product with $A^k(e)$, $0\le k\le N-1$, we get

$$(A^{N+k}(e),e)=\sum_{j=1}^N(-1)^{j-1}\sigma_j (A^{N-j+k}(e),e).$$
Solving the corresponding system in $(\sigma_1,\dots ,\sigma _N)$, we get that $\sigma _1=\text{Tr}(\tilde A)$ coincides with the right hand side of the inequality. Hence, inequality of lemma \ref{traceA} is proved. 
Furthermore, there is equality if and only if $$\text{Tr}((I-P)A(I-P))=0.$$ 
This is equivalent, since $A$ is  positive, to $(I-P)A(I-P)=0$ which, in turn is equivalent to $A(I-P)=0$. Indeed, let $w\in \text{Im}(I-P)$ so that $(I-P)Aw=0$ then $((I-P)Aw,w)=0=(Aw,w)$ so that $Aw=0$. In particular, the range of $A$ is a subspace of $V$.  On the other hand, by assumption the range of $A$ is at least $N$ dimensional, we obtain that the range of $A$ is exactly $V$. In particular, it implies that $e$ belongs to the range of $A$. Conversely, if the range of $A$ is $N$ dimensional and if $e$ belongs to the range of $A$, then $V$ is a subspace of the range of $A$ and is $N$ dimensional, hence $V$ is the range of $A$. In particular, $(I-P)A=0$ so that $\text{Tr}(\tilde A)=\text{Tr}(A)$.
\end{proof}
We now show that Proposition \ref{M(u)det} implies the theorem, namely that ${\bf T}$ is the solution of the minimization problem. It remains to prove that, if $u\in \mathcal M(N)$ satisfies $J_{2n}(u)=j_{2n}$ for $n=1,\dots ,2N$, then $u\in {\bf T}$. Let $u$ be such a function. Since $\det (J_{2(n+m)}(u))_{1\le n,m\le N}=\det (j_{2(n+m)})_{1\le n,m\le N}\ne 0$, we already know that $H_u^2$ has $N$ simple positive eigenvalues 
$\tilde {\lambda }_1^2>\dots >\tilde{\lambda }_N^2$ and its corresponding  normalization constants $\tilde {\nu }_1,\dots ,\tilde{\nu }_N$ are all $>0$. Let us  prove that
$$\tilde{\lambda }_j=\lambda _j\ ,\ \tilde{\nu }_j=\nu _j \ ,\  j=1,\dots ,N$$
where $\lambda _1,\dots ,\lambda_N$, $\nu _1,\dots ,\nu _N$ correspond to any element $u_0\in {\bf T}$. The assumption $J_{2n}(u)=j_{2n}$ for $n=1,\dots ,2N$ reads
$$\sum_{j=1}^N\tilde {\lambda }_j^{2n}\, \tilde{\nu }_j^2=\sum_{j=1}^N\lambda _j^{2n}\nu _j^2\ ,\ n=1\dots, 2N\ ,$$
or, for every polynomial $P$ of degree $\leq 2N$ such that $P(0)=0$,
$$\sum _{j=1}^NP(\tilde{\lambda }_j^2)\tilde{\nu }_j^2=\sum _{j=1}^NP(\lambda _j^2)\nu _j^2\ .$$
Assume that for some $j_0$, $\tilde \lambda _{j_0}$ is different from all the $\lambda _j$'s. Then we can select a polynomial $P$ of degree $2N$ such that $P(\lambda _j^2)=0$
for every $j$, $P(\tilde{\lambda }_j^2)=0$ for every $j\ne j_0$ and $P(0)=0$, but $P(\tilde{\lambda }_{j_0}^2)\ne 0$. Plugging these informations into the above identity, we get
$\tilde \nu _{j_0}=0$, a contradiction. This implies $\tilde \lambda _j=\lambda _j$ for every $j$, and finally, by solving a Van der Monde system, $\tilde \nu _j=\nu _j$ for every $j$.
\end{proof}

\begin{remark}\label{stabilityinfini}
There is an analogous result of Theorem \ref{stability} in the infinite dimensional case, though it is easier. Indeed, given two sequences $I=(I_j)_{j\ge 1}, L=(L_m)_{m\ge 1}$ of numbers such that
$$I_1>L_1>I_2>L_2>\dots >0\ ,\ \sum _{j=1}^\infty I_j<\infty \ ,$$
denote by ${\bf T}(I,L)$ the infinite dimensional torus  of those $u\in H^{1/2}_{+,{\rm gen}}$ such that $\chi (u)=((\zeta _j)_{j\ge 1}, (\gamma _m)_{m\ge 1}) $ with $I_j=2\vert \zeta _j\vert ^2$ and $L_m=2\vert \gamma _m\vert ^2$ for all $j,m$. First of all, we observe that, for every $n\ge 1$, $J_{2n}$ has a constant value on ${\bf T}(I,L)$  given by
$$j_{2n}=\sum _{j=1}^\infty 2^{-n}I_j^n\left (1-\frac {L_j}{I_j}\right )\prod _{k\ne j}\left (\frac{L_k-I_j}{I_k-I_j}\right )\ .$$
Then we claim that ${\bf T}(I,L)$ is precisely the solution of the minimization problem
$$\inf \{ M(u)\;  : \; J_{2n}(u)=j_{2n}\ ,\ n\ge 1\} \  .$$
Indeed, if $u\in H^{1/2}_+$ is such that $J_{2n}(u)=j_{2n}$ for every $n\ge 1$, we conclude that
$$\forall x\notin \left \{ \frac 1{\lambda _j^2}\right \}\ ,\ J(x)(u)=\prod _{j=1}^\infty \frac{1-\mu _j^2x}{1-\lambda _j^2x}\ ,$$
where $$\lambda _j^2:=\frac 12 I_j\ ,\  \mu _j^2:=\frac 12 L_j\ .$$
From formula (\ref{ProdHmu}), we infer
\begin{equation}\label{prodprod}
\prod _{j=1}^\infty \frac{1-\mu _j^2(u)x}{1-\lambda _j^2(u)x}=\prod _{j=1}^\infty \frac{1-\mu _j^2x}{1-\lambda _j^2x}\ .
\end{equation}
Consequently, the sequence $(\lambda _j^2)$ is a subsequence of the sequence $(\lambda _j^2(u))$, and the sequence $(\mu _j^2)$ is a subsequence of the sequence $(\mu _j^2(u))$.
We deduce 
$$M(u)={\rm Tr}(K_u^2)=\sum _{m=1}^\infty \mu _m^2(u)\ge \sum _{m=1}^\infty \mu _m^2\ ,$$
with equality if and only if the sequences $(\mu _j^2)$ and $(\mu _j^2(u))$ coincide. In that case, in view of (\ref{prodprod}), we conclude that the sequences $(\lambda _j^2)$ and $(\lambda _j^2(u))$
coincide too, and finally that $u\in {\bf T}(I,L)$. The stability of ${\bf T}(I,L)$ through the evolution of (\ref{szego}) therefore follows by the same compactness arguments as in the proof of Theorem \ref{stability}.
\end{remark}
\end{section}

\begin{section}{Instability of traveling waves}\label{instability}

In contrast with the preceding section, we now establish instability of traveling waves which are non minimal.

\begin{theorem} 
The following traveling waves of the cubic Szeg\"o equation are orbitally unstable :
\begin{eqnarray*}
\varphi (z)&=&\alpha \prod _{j=1}^N\frac{z-\overline p_j}{1-p_jz}\ ,\ \alpha \ne 0, N\ge 1\ ,\  0\le \vert p_j\vert <1\ ,\\
\varphi (z)&=&\alpha \frac {z^\ell }{1-p^Nz^N}\ ,\ \alpha \ne 0\ ,\ N\ge 2, N-1\ge \ell \ge 0\ ,\  0< \vert p\vert <1\  .
\end{eqnarray*}
\end{theorem}
\begin{proof}
We first deal with traveling waves with non zero velocity,
$$\varphi (z)=\frac {z^\ell }{1-p^Nz^N}\ ,\ N\ge 2, N-1\ge \ell \ge 0\ ,\  0< \vert p\vert <1\  ,$$
where the constant $\alpha \ne 0$ has been made $1$ for simplicity, in view of the invariances of the equation.
Our strategy is to approximate $\varphi $ by a family $(u_0^\e )$ in $\mathcal M(N)_{\rm gen}$ such that the family 
of corresponding solutions  
$(u^\e )$ do not satisfy
\begin{equation}\label{orbitstab}
\sup _{t\in \R}\inf _{(\alpha ,\beta )\in \T^2}\Vert u^\e (t)-\varphi _{\alpha ,\beta }\Vert _{H^{1/2}}\td _\e ,0  0\ ,
\end{equation}
where $\varphi _{\alpha ,\beta }$ denotes the current point of the orbit of $\varphi $ through the action of $\T ^2$,
$$\varphi _{\alpha ,\beta}(z)={\rm e}^{i\alpha }\varphi ({\rm e}^{i\beta}z)\ .$$
Specifically, if (\ref{orbitstab}) holds, then, for $\e $ small enough, $u^\e (t)$ belongs to a compact subset 
of $\mathcal M(N)$, and consequently every continuous fonction $f$ on $\mathcal M(N)$ which vanishes on 
every $\varphi _{\alpha ,\beta}$ satisfies
$$\sup _{t\in \R}\vert f(u^\e (t))\vert \td _\e ,0  0\ .$$
We shall choose for $f$, the function $\sigma $ defined by
$$u(z)=\frac{A(z)}{1-\sigma (u)z+z^2R(z)}\ ,$$
where $A, R$ are polynomial functions.
Notice that $\sigma $ vanishes onto the orbit of $\varphi $ since $N\ge 2$. We now compute $\sigma (u^\e (t))$ by means of the explicit inverse formula for $\chi _N$ given in Proposition \ref{inversespectral}.
This yields
$$\sigma (u^\e (t))={\rm tr}(\Gamma )=\sum _{1\le j,\ell\le N}\frac{\lambda _j\nu _j^2\mu _\ell}{b_\ell (\lambda _j^2-\mu _\ell^2)^2}\, {\rm e}^{-i(\varphi _j+\theta _\ell)}\ .$$
Notice that, in the above formula, all the quantities depend on $\e $, but only the angles $\varphi _j,\theta _\ell $ depend on $t$.
Moreover, from Corollary \ref{szegosolution}, we know that they depend linearly on $t$, with velocities
$$\frac d{dt}(\varphi _j+\theta _\ell )=\lambda _j^2-\mu _\ell ^2.$$
We claim that we may assume that all these velocities are pairwise distinct. Indeed,  using the diffeomorphism $\chi _N$ of Theorem \ref{TheoDiffeo}, this just comes from the fact that, on the open set $$\Omega _N=\{ I_1>L_1>I_2>L_2>\dots >I_N>L_N>0\} $$ of $\R^{2N}$, the quantities $I_j-L_\ell $ are generically pairwise distinct. 
Consequently, 
$$\frac 1T\int _0^T\vert \sigma (u^\e (t))\vert ^2\, dt \td _T,\infty \sum _{1\le j,\ell\le N}\frac{\lambda _j^2\nu _j^4\mu _\ell ^2}{b_\ell ^2(\lambda _j^2-\mu _\ell^2)^4}\ .$$
We now estimate the right hand side of the above identity from below as $\e $ tends to $0$. An elementary spectral study of $H_\varphi ^2$ and of $K_\varphi ^2$ 
shows that their eigenvectors are 
$$\varphi _j=\frac{z^j}{1-p^Nz^N}\ ,\ j=0,1,\dots ,N-1,$$
and that their eigenvalues belong to the pair
$$\left\{ \frac{\vert p\vert ^{2N}}{(1-\vert p\vert ^{2N})^2}, \frac{1}{(1-\vert p\vert ^{2N})^2}\right \}$$
hence are bounded from above and below. From the continuity deduced from the min max formula, we infer that $\lambda _j^2, \mu _\ell ^2$ are also bounded
from above and below as $\e $ tends to $0$. Therefore, for some fixed positive constant $\delta $,
$$\sum _{1\le j,\ell\le N}\frac{\lambda _j^2\nu _j^4\mu _\ell ^2}{b_\ell ^2(\lambda _j^2-\mu _\ell^2)^4}\ge \delta \left (\sum _{j=1}^N\nu _j^4\right )\left (\sum _{\ell =1}^N \frac {1}{b_\ell^2}\right )\ge 
\frac{\delta }{N^2}\left (\sum _{j=1}^N\nu _j^2\right )^2\left (\sum _{\ell =1}^N \frac {1}{b_\ell}\right )^2.$$
Moreover,
$$\sum _{\ell =1}^N \frac {1}{b_\ell}=\Vert u_0^\e \Vert ^2\td _\e ,0 \Vert \varphi \Vert ^2\ ,\  \sum _{j=1}^N\nu _j^2=\Vert P_{u^\e_0}(1)\Vert ^2\td _\e ,0 \Vert P_{\varphi }(1)\Vert ^2=1-\vert p\vert ^{2N}\ ,$$
as we proved in \cite{GG}, Proposition 1.
We conclude that
$$\liminf _{\e \rightarrow 0}\lim _{T\rightarrow \infty}\frac 1T\int _0^T\vert \sigma (u^\e (t))\vert ^2\, dt >0\ ,$$
which contradicts
$$\sup _{t\in \R}\vert \sigma (u^\e (t))\vert \td _\e ,0  0\ .$$
Hence $\varphi$ is orbitally unstable.
\s
We now deal with stationary waves, which are Blaschke products
$$\varphi (z)=\prod _{j=1}^{N-1} \frac{z-\overline p_j}{1-p_jz}$$
with $N\ge 2, 0\le \vert p_j\vert <1$. Once again, we want to prove that
there exists a sequence $u_0^\e $ in $H^{1/2}_+$ such that
$$\Vert u_0^\e -\varphi \Vert _{H^{1/2}}\td _\e ,0  0$$
but the solution $u^\e $ of the cubic Szeg\"o equation with Cauchy datum $u_0^\e $ satisfies
\begin{equation}\label{instabla}
\liminf _{\e\rightarrow 0}\sup _{t\in \R }\inf _{\alpha \in \T }\Vert u^\e (t)-{\rm e}^{i\alpha }\varphi \Vert_{H^{1/2}}>0 \ .
\end{equation}
Introduce the quantity
$$q:=(1\vert \varphi )=(-1)^{N-1}p_1\dots p_{N-1}\ .$$
We claim that we may assume that $q\in \R _+$. Indeed, by using invariance of the cubic Szeg\"o equation through 
multiplication by complex numbers of modulus $1$ and by rotations of the circle, property (\ref{instabla}) for $\varphi $ and the sequence $(u_0^\e )$ 
is equivalent to 
property (\ref{instabla}) for 
$$\tilde \varphi (z)={\rm e}^{i\beta}\varphi ({\rm e}^{i\gamma }z)$$
and $$\tilde u_0^\e (z)={\rm e}^{i\beta}u_0^\e ({\rm e}^{i\gamma }z)\ .$$
If we choose $\beta =-(N-1)\gamma $, we observe that
$$\tilde \varphi (z)=\prod _{j=1}^{N-1} \frac{z-\overline p_j'}{1-p_j'z}\ ,\ p_j':= {\rm e}^{i\gamma }p_j\ ,$$
so that $\tilde q={\rm e}^{i(N-1)\gamma }q$. Hence a convenient choice of $\gamma $ ensures $\tilde q\ge 0$.
We therefore assume from now on that $q\ge 0$.

We now introduce
$$u_0^\e =\varphi +\e \ .$$
Let us first determine the spectrum of $H_{u_0^\e }^2$ on the vector space $$<1>:={\rm span} (H_{u_0^\e}^n(1), n\ge 0).$$ We have
$$H_{u_0^\e }=H_\varphi +\e H_1$$
which  is identically $0$ on $\ker H_\varphi $. On the range of $H_\varphi $, we have
$$H_{u_0^\e }^2=H_\varphi ^2+\e (H_\varphi H_1+H_1H_\varphi)+\e ^2H_1^2 =I+R_\e $$
where $R_\e $ is the rank two operator defined by
$$R_\e (h)=\e ((h\vert 1)\varphi +(h\vert \varphi ))+\e ^2(h\vert 1).$$
Notice that we used the identity $H_\varphi ^2=1$ on the range of $H_\varphi $, which holds since $\varphi $ is an inner function.
We observe that $R_\e $ stabilizes ${\rm span} (1,\varphi )$, which is therefore $<1>, $
and that its matrix  in the basis $(1,\varphi )$ reads
\begin{eqnarray*}
M_\e =\left (\begin{array}{cc}\e q+\e ^2& \e +\e ^2 q\\ 
\e &\e q 
\end{array}
\right )
\end{eqnarray*} 
Consequently, the eigenvalues of $H_{u_0^\e }^2$ on   ${\rm span} (1,\varphi )$ are the
roots $r_\pm $ of the equation
$$(r-1)^2-(2\e q+\e^2)(r-1)-\e ^2(1-q^2)=0$$
which are given by
$$r_\pm =1+\e \left (q+\frac{\e }{2}\pm (1+\e q +\frac{\e ^2}4)^{1/2} \right )=1+\e (q\pm 1)+O (\e ^2)\ .$$
We therefore have
\begin{eqnarray*}
H_{u_0^\e }^4(1)-\sigma _1H_{u_0^\e }^2(1)+\sigma _2&=&0\ ,\\
 \sigma _1=r_++r_-=2+2\e q+\e ^2\ ,\ \sigma _2&=&r_+r_-=(1+\e q)^2\ .
\end{eqnarray*}
Applying the Lax pair property described in Formulae (\ref{laxpair}) and (\ref{Bu}), we infer
$$H_{u^\e }^4(1)-\sigma _1H_{u^\e }^2(1)+\sigma _2=0$$
for every time $t$. Indeed, $f=H_{u^\e }^4(1)-\sigma _1H_{u^\e }^2(1)+\sigma _2$ satisfies the linear evolution equation
$$\frac {df}{dt}=(B_{u^\e }+\frac i2H_{u^\e }^2)f$$
and $f(0)=0$. Denote by $w^\e \in <1>$ the unique vector such that $H_{u^\e}(w^\e )=1.$
In view of the above  formula, we have
$$w^\e =\frac{-H_{u^\e }^3(1)+\sigma _1H_{u^\e}(1)}{\sigma _2}\ ,$$
We now study the evolution of
$$J_{-1}^\e (t)=(w^\e \vert 1)\ ,\ J_1^\e (t)=(u^\e \vert 1)\ .$$
Again by the Lax pair property, we have
$$i\dot J_{-1}^\e =J_1^\e \ ,\ i\dot J_1^\e =\sigma _1J_1^\e -\sigma _2J_{-1}^\e \ ,$$
which implies that
$$J_1^\e (t)=\gamma _+ {\rm e}^{-ir_+t}+\gamma _- {\rm e}^{-ir_-t}\ ,$$
where $\gamma _\pm $ are given by initial conditions,
$$\gamma _++\gamma _-=q+\e \ ,\ \frac{\gamma _+}{r_+}+\frac{\gamma _-}{r_-}=\frac{q}{1+\e q}\ .$$
This leads to
$$\gamma _\pm = \frac{q\pm 1}2 +O(\e )\ .$$
We infer, for every $s>0$,
\begin{eqnarray*}\frac \e s\int _0^{s/\e }\vert J_1^\e (t)\vert ^2\, dt =\frac 1s\int _0^s( \gamma _+^2+\gamma _- ^2+2{\rm Re}(\gamma _+\gamma _-{\rm e}^{-i(r_+-r_-)\sigma /\e }))\, d\sigma \\
\td _\e ,0   \frac{1+q^2}{2} -\frac{(1-q^2)\sin (2s)}{4s} :=f(s)\ .
\end{eqnarray*}
On the other hand, if $\varphi $ is orbitally stable, we  have 
$$\sup _{t\in \R }\vert \vert J_1^\e (t)\vert ^2-q^2 \vert \le C\sup _{t\in \R }\inf _{\alpha \in \T }\Vert u^\e (t)-{\rm e}^{i\alpha }\varphi \Vert _{L^2}\td _\e ,0 0\ ,$$
which imposes $f(s)=q^2$ for every $s$ and  contradicts the above formula for $f$.
We conclude that $\varphi $ is orbitally unstable.

 \end{proof}
\end{section}

\end{document}